\documentclass[11pt,english]{article}
 \usepackage[english]{babel}
\usepackage{cite} 

\usepackage{graphicx,subfigure}

\usepackage[usenames, dvipsnames]{xcolor}
\usepackage[utf8x]{inputenc}
\usepackage[T1]{fontenc}
\usepackage{tikz}
\usepackage{pgfplots}

\usepackage{amsmath}
\usepackage{amsfonts}
\usepackage{amssymb}
\usepackage{amsthm}

\usepackage{esint}

\usepackage[colorlinks=true,urlcolor=blue,
citecolor=red,linkcolor=blue,linktocpage,pdfpagelabels,
bookmarksnumbered,bookmarksopen,backref=page]{hyperref}

\DeclareMathOperator{\sech}{sech}
\DeclareMathOperator{\Span}{span}

 \newcommand{\RE}{\mathrm{Re}\,}
 
\newtheorem{theorem}{Theorem}[section]

 \newtheorem{proposition}[theorem]{Proposition}
 \theoremstyle{definition}
 \newtheorem{definition}[theorem]{Definition}
 \theoremstyle{remark}
 \newtheorem{remark}[theorem]{Remark}
 
 \numberwithin{equation}{section}

\newcommand{\R}{\mathbb R}
\begin{author}
{Jaime Angulo Pava $^1$ and Ram\'on G. Plaza $^2$}
\end{author}

\begin{title}{Instability theory of kink and anti-kink profiles  for the sine-{G}ordon equation on Josephson tricrystal boundaries}
\end{title}
\begin{document}
\maketitle

\centerline{$^1$ Department of Mathematics,
IME-USP}
 \centerline{Rua do Mat\~ao 1010, Cidade Universit\'aria, CEP 05508-090,
 S\~ao Paulo, SP (Brazil)}
 \centerline{\tt angulo@ime.usp.br}
 \centerline{ $^2$ Instituto de Investigaciones en Matem\'aticas Aplicadas
   y en Sistemas,}
 \centerline{Universidad Nacional Aut\'{o}noma de M\'{e}xico,  Circuito Escolar s/n,}
 \centerline{Ciudad Universitaria, C.P. 04510 Cd. de M\'{e}xico (Mexico)}
 \centerline{\tt  plaza@mym.iimas.unam.mx}

\begin{abstract}

The aim of this work is to establish a  instability study for stationary kink and antikink/kink profiles solutions for the sine-Gordon equation on a metric graph with a structure represented by a $\mathcal Y$-junction so-called a Josephson  tricrystal junction. By considering boundary conditions at the graph-vertex of $\delta'$-interaction type, it is shown that these kink-soliton type stationary profiles are linearly (and nonlinearly) unstable. The extension theory of symmetric operators, Sturm-Liouville oscillation results and analytic perturbation theory of operators are fundamental ingredients in the stability analysis. The local well-posedness of the sine-Gordon model in $ H^1(\mathcal Y) \times L^2(\mathcal{Y})$  is also established. The theory developed in this investigation has prospects for the study of the (in)-stability of stationary wave solutions of other configurations for kink-solitons profiles. 
\end{abstract}


\textbf{Key  words}.  sine-Gordon model, Josephson tricrystal junction, $\delta'$-type interaction, kink, anti-kink solitons, perturbation theory, extension theory, instability.


\section{Introduction}

Nonlinear dispersive models on quantum star-shaped graphs (quantum graphs, henceforth) arise as simplifications for wave propagation, for instance, in quasi one-dimensional (e.g. meso- or nano-scaled) systems that {\it look like a thin neighborhood of a  graph}. We recall that a star-shaped metric graph, $\mathcal G$, is a structure represented by a finite or countable edges attached to a common vertex, $\nu=0$, having each edge identified with a copy of the half-line, $(-\infty, 0)$ or $(0, +\infty)$ (see Figure 1 below). Hence, a quantum star-shaped metric graph, $\mathcal G$, is a  star-shaped metric graph with a linear Hamiltonian operator (for example, a Schr\"odinger-like operator) suitably defined on functions which are supported on the edges. 

Quantum graphs have been used to describe a variety of physical problems and applications, for instance, condensed matter,  $\mathcal{Y}$-Josephson junction networks, polymers, optics, neuroscience, DNA chains, blood pressure waves in large arteries, or in shallow water models describing a fluid networks (see \cite{Berko17,BurCas01,Caput14,Fid15,Mug15} and the many references therein).  Recently, they have attracted much attention in the context of soliton transport in networks and branched structures since wave dynamics in networks can be modeled by nonlinear evolution equations (see, e.g., \cite{AdaNoj16, AngCav2, AngGol18b, AngGol18a,AnPl-delta,AnPl-deltaprime,MNS,SvG05,Susa19}).

The present study focuses on the dynamics of the one-dimensional  sine-Gordon equation,
\begin{equation}
\label{sine-G}
u_{tt} - c^2 u_{xx} + \sin u = 0,
\end{equation}
posed on a metric graph. The sine-Gordon model appears in a great variety of physical and biological models. For example, it has been used to describe the magnetic flux in a long Josephson line in superconductor theory \cite{BEMS,BaPa82,SCR}, mechanical oscillations of a nonlinear pendulum \cite{Dra83,Knob00} and the dynamics of a crystal lattice near a dislocation \cite{FreKo}. Recently, soliton solutions to equation \eqref{sine-G} have been used as simplified models of scalar gravitational fields in general relativity theory \cite{CFMT19,FCT18} and of oscillations describing the dynamics of DNA chains \cite{DerGa11,IvIv13} in the context of the \emph{solitons in DNA hypothesis} \cite{EKHKL80}. In addition, the sine-Gordon equation \eqref{sine-G} underlies many remarkable mathematical features such as a Hamiltonian structure \cite{TaFa76}, complete integrability \cite{AKNS73,AKNS74} and the existence of localized solutions (solitons) \cite{SCM,Sco03}. 

In recent contributions (cf. \cite{AnPl-delta, AnPl-deltaprime}), we performed the first rigorous analytical studies of the stability properties of stationary soliton solutions of kink and/or anti-kink profiles to the sine-Gordon equation \emph{posed on a $\mathcal{Y}$-junction graph}. There exist two main types of $\mathcal{Y}$-junctions. A $\mathcal{Y}$-junction of the first type (or type I) consists of one incoming (or parent) edge, $E_1 = (-\infty,0)$, meeting at one single vertex at the origin, $\nu = 0$, with other two outgoing (children) edges, $E_j = (0,\infty)$, $j = 2,3$. The second type (or $\mathcal{Y}$-junction of type II) resembles more a starred structure and consists of three identical edges of the form $E_j = (0,\infty)$, $1 \leqq j \leqq 3$. See Figure \ref{figYjunction} for an illustration. Junctions of type I are more common in unidirectional fluid flow models (see, for example, \cite{BoCa08}). The junctions of the second type (which belong to the star graph class) are often referred as \emph{Josephson tricrystal junctions} and are common in network of transmission lines (see, for instance, \cite{SvG05,KCK00,Sabi18}).

\begin{figure}[t]
\begin{center}
\subfigure[$\mathcal{Y} = (-\infty,0) \cup (0,\infty) \cup (0,\infty)$]{\label{figYtipoI}\includegraphics[scale=.4, clip=true]{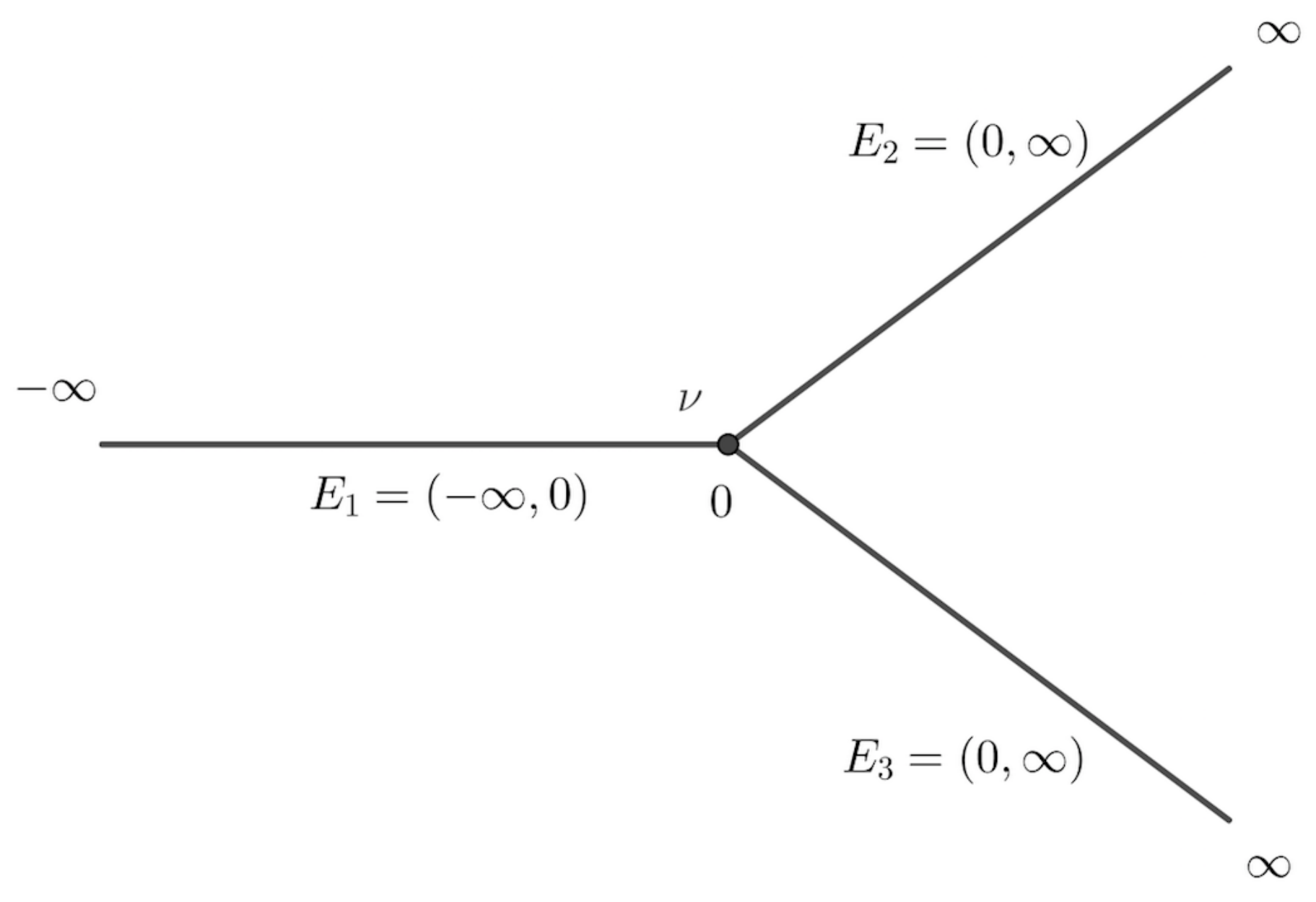}}
\subfigure[$\mathcal{Y} = (0,\infty) \cup (0,\infty) \cup (0,\infty)$]{\label{figYtipoII}\includegraphics[scale=.4, clip=true]{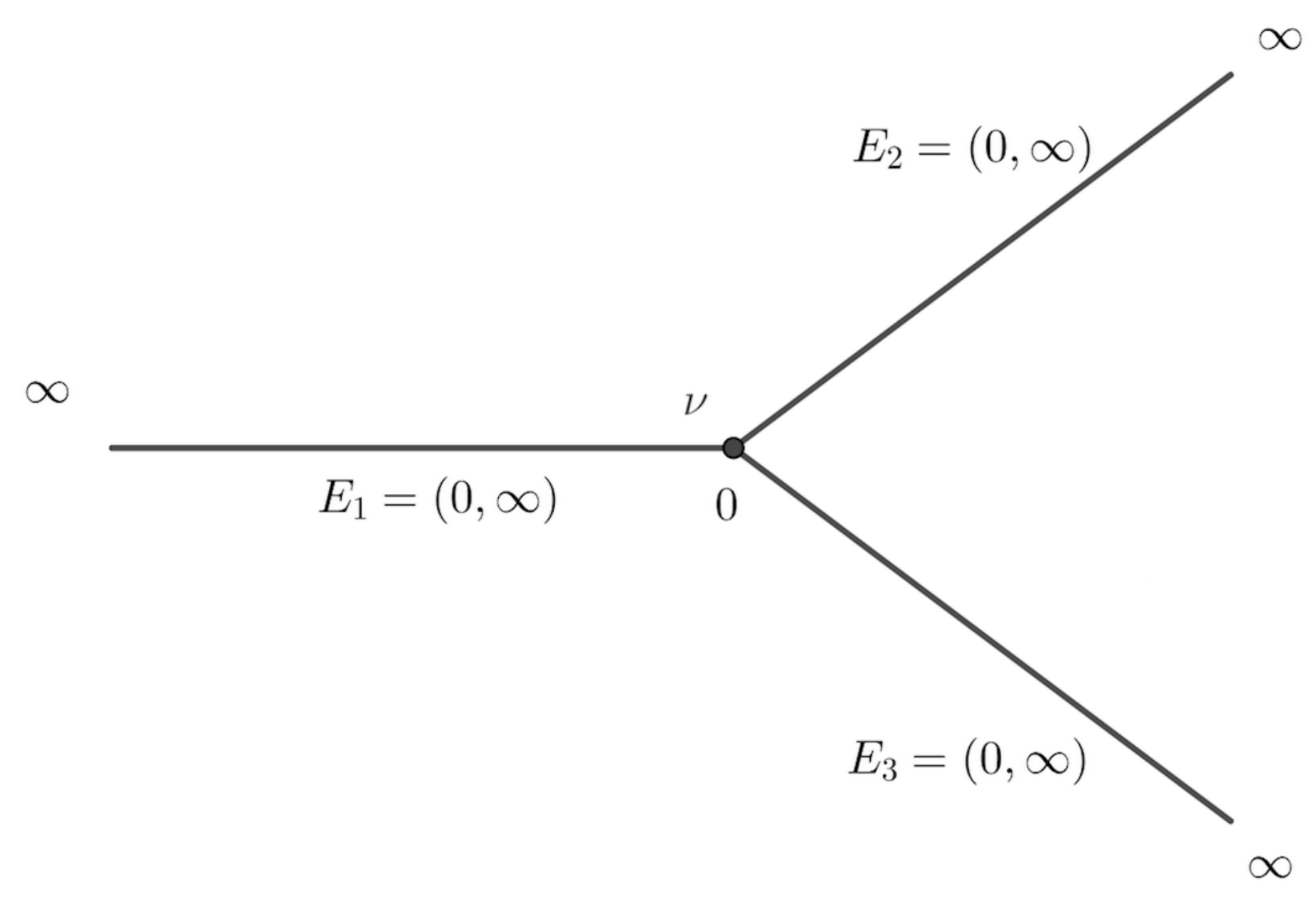}}
\end{center}
\caption{\small{Panel (a) shows a $\mathcal{Y}$-junction of the first type with $E_1 = (-\infty,0)$ and $E_j = (0,\infty)$, $j=2,3$, whereas panel (b) shows a $\mathcal{Y}$-junction of the second type (star graph) with $E_j = (0,\infty)$, $1 \leqq j \leqq 3$.}}\label{figYjunction}
\end{figure}

Our work focuses on the sine-Gordon model posed on a $\mathcal Y$-tricrystal junction, more preciely, on the equations
\begin{equation}
\label{sg1}
\partial_t^2 u_j - c_j^2 \partial_x^2 u_j + \sin u_j= 0, \qquad x \in E_j=(0, \infty), \;\,  t > 0, \;\, 1 \leqq j \leqq 3,
\end{equation}
where $u_j = u_j(x,t)$. It is assumed that the characteristic speed on each edge $E_j$ is constant and positive, $c_j > 0$, without loss of generality.

Posing the sine-Gordon equation on a metric graph comes out naturally from practical applications. Indeed, in the context of superconductor theory, the sine-Gordon equation on a metric graph arises as a model for coupling of two or more Josephson junctions in a network. A Josephson junction is a quantum mechanical structure that is made by two superconducting electrodes separated by a barrier (the junction), thin enough to allow coupling of the wave functions of electrons for the two superconductors \cite{Jsph65}. After appropriate normalizations, it can be shown that the phase difference $u$ (also known as order parameter) of the two wave functions satisfies the sine-Gordon equation \eqref{sine-G} \cite{Jsph65,BEMS}. Coupling three junctions at one common vertex, the so called tricrystal junction, can be regarded (and fabricated) as a probe of the order parameter symmetry of high temperature superconductors (cf. \cite{Tsuei94,Tsuei00}). Physically coupling three otherwise independent long Josephson junctions, $\mathcal{Y} = \cup_{j=1}^3 E_j$, together at one common vertex, was first proposed by Nakajima \emph{et al.} \cite{NakO76,NakO78} as a prototype for logic circuits.

What is more crucial in the analysis on quantum graphs is the choice of boundary conditions, mainly because the transition rules at the vertex completely determine the dynamics of the PDE model on the graph. For the sine-Gordon equation in $\mathcal{Y}$-junctions,  previous studies have basically (and almost exclusively) considered two types of boundary conditions: interactions of $\delta$ (continuity of the wave functions plus a law of Kirchhoff-type for the fluxes at the vertex, see \cite{AnPl-delta, Caput14, DuCa18} ) and of $\delta'$-type (continuity of the derivatives plus a Kirchhoff law for the self-induced magnetic flux).  Since Josephson models arise in the description of electromagnetic flux, interactions of $\delta'$-type have received more attention (see, for example, \cite{AnPl-deltaprime, Grunn93, KCK00, NakO78, NakO76, SvG05}). Thus, since the surface current density should be the same in all three films at the vertex, Nakajima \emph{et al.} \cite{NakO76,NakO78} (see also \cite{Grunn93,KCK00}) impose the condition
\begin{equation}
\label{bcmagnet}
c_1 \partial_x {u_1}_{|x=0} = c_2 \partial_x {u_2}_{|x=0} = c_3 \partial_x {u_3}_{|x=0},
\end{equation}
expressing that the magnetic field, which is proportional to the derivative of phase difference, should be continuous at the intersection. Moreover, the magnetic flux computed along an infinitesimal small contour encircling the origin (vertex) must vanish, that is, the total change of the gauge invariant phase difference must be zero \cite{KCK00,Susa19}. This leads to the Kirchhoff-type of boundary condition
\begin{equation}
\label{Kirchhoffbc}
\sum_{j=1}^3 c_j  u_j(0+) =0.
\end{equation}
The interaction conditions \eqref{bcmagnet}-\eqref{Kirchhoffbc} are known as boundary conditions of $\delta'$-type: they express continuity of the fluxes (derivatives) plus a Kirchhoff-type rule for the self-induced magnetic flux. 

Motivated by physical applications, the purpose of the present paper is to study the stability of particular stationary solutions to the sine-Gordon equation seen as a first order system posed on a $\mathcal{Y}$-tricrystal junction, namely, the system
\begin{equation}
\label{sg2}
\begin{cases}
\partial_t u_j = v_j\\
\partial_t v_j =c_j^2 \partial_x^2 u_j - \sin u_j,
\end{cases}
\qquad x \in E_j=(0, \infty), \;\,  t > 0, \;\, 1 \leqq j \leqq 3. 
\end{equation}

As far as we know, there is no rigorous analytical study of the stability of stationary solutions of type kink and/or anti-kink to the vectorial sine-Gordon model \eqref{sg2} on a tricrystal junction with boundary conditions of $\delta'$-interaction type available in the literature (see in Angulo and Plaza \cite{AnPl-deltaprime} the case of  $\delta'$-interactions on $\mathcal Y$-junction graph of type I). The stability of these static configurations is an important property from both the mathematical and the physical points of view. Stability can predict whether a particular state can be observed in experiments or not. Unstable configurations are rapidly dominated by dispersion, drift, or by other interactions depending on the dynamics, and they are practically undetectable in applications.

For tricrystal junctions  we will study stationary solutions with a kink  type structure. The latter have the form $
u_j (x,t) = \varphi_j(x)$, $v_j(x,t) = 0$,
for all $j = 1,2,3$, and $x \in E_j$, $t > 0$, in \eqref{sg2}, where each of the profile functions $\varphi_j$ satisfies the equation 
\begin{equation}\label{kinkequa}
-c_j^2 \varphi''_j + \sin \varphi_j=0,
\end{equation}
on each edge $E_j=(0, \infty)$ and for all $j$, as well as the boundary conditions of $\delta'$-type, at the vertex $\nu = 0$:
\begin{equation}
\label{bcIIfam}
\begin{aligned}
c_1 \varphi'_1(0+) = c_2 \varphi'_2(0+) &= c_3 \varphi'_3(0+),\\
\sum_{j=1}^3 c_j  \varphi_j(0+) &= \lambda c_1 \varphi'_1(0+).
\end{aligned}
\end{equation}
These conditions depend upon the parameter $\lambda$, which ranges along the whole real line and determines the dynamics of the model. Therefore, the value $\lambda \in \R$ is part of the physical parameters that determine the model (such as the speeds $c_j$, for instance). Instead of adopting \emph{ad hoc} boundary conditions, we consider a parametrized family of transition rules covering a wide range of applications and which, for the particular value $\lambda = 0$, include the Kirchhoff condition \eqref{Kirchhoffbc} previously studied in the literature. Our analysis focuses on two particular class of solutions of the sine-Gordon equation known as \emph{kink} and \emph{anti-kink}  (also referred to as topological solitons) \cite{Dra83,Sco03,SCM}.

Initially, we look at the particular family of solutions with the same kink-type structure with $c_j>0$,
\begin{equation}\label{tricry1u}
\varphi_j(x)= 4 \arctan \big( e^{-(x-b_j)/c_j}\big), \qquad  x \in (0,\infty), \,\; j=1,2,3,\\
\end{equation}
where the constants $b_j$ are determined by the boundary conditions \eqref{bcIIfam}. The family satisfies as well 
\begin{equation}
\label{bcinfty}
\varphi_j(+\infty) = 0, \qquad j = 1,2,3,
\end{equation}
ensuring that $\varPhi = (\varphi_j)_{j=1}^3 \in H^2(\mathcal{Y})$. Our second class of  solutions   are the anti-kink/kink type soliton, namely,
 profiles having the form 
\begin{equation}
\label{antikink}
\begin{cases}
\varphi_1(x) = 4 \arctan \big( e^{(x-a_1)/c_1}\big), & x \in (0, \infty),\;\; \lim_{x\to +\infty}\varphi_1(x)=2\pi,\\
\varphi_j(x) = 4 \arctan \big( e^{(x-a_j)/c_j}\big)-2\pi, & x \in (0,\infty),\;\;  \lim_{x\to +\infty}\varphi_j(x)=0\;\; j=2,3,\\
\end{cases}
\end{equation}
where each $a_j$ is a constant determined by the boundary conditions \eqref{bcIIfam}. Expression in \eqref{antikink} is so-called $2\pi$-kink because the total Josephson phase is $2\pi$ when one circles the branch point at large distances, i.e., 
$\sum_{j=1}^3 \varphi_j(+\infty)=2\pi$. We have left open the stability study of other Josephson configurations systems, such as a tricrystal junction with a $\pi$-kink, $\phi_1(x) = 4 \arctan \big( e^{(x-a_1)/c_1}\big)-\pi$,  $\phi_j(x) = 4 \arctan \big( e^{(x-a_j)/c_j}\big)-2\pi$, $j=2,3$. This stability analysis is of important also for experimentalists since
these network systems open a large opportunity for applications in high-performance computers (see \cite{SvG05} and references therein).

In the forthcoming analysis we establish the existence of two smooth mapping of stationary profiles, the first one  $\lambda\in  ( -\infty, -\sum_{j=1}^3 c_j)\to \Pi_{\lambda,\delta'}=(\varphi_1,\varphi_2,\varphi_3, 0,0,0)$, with $\varphi_j=\varphi_{j, b_j(\lambda)}$ defined in \eqref{tricry1u}, and the second one $\lambda\in  ( -\infty, +\infty)\to \Phi_{\lambda,\delta'}=(\varphi_1, \varphi_{2}, \varphi_{3}, 0,0,0)$, with $\varphi_j= \varphi_{j, a_j(\lambda)}$ defined in \eqref{antikink}.

The main linear instability results of this manuscript are the following,

\begin{theorem}
\label{trimain} 
Let $\lambda\in  ( -\infty, -\sum_{j=1}^3 c_j)$ and consider the smooth family of stationary profiles $\lambda\mapsto \Pi_{\lambda,\delta'}$. Then $\Pi_{\lambda, \delta'}$ is linear and nonlinearly  unstable for the sine-Gordon model \eqref{stat1} on a tricrystal junction  in the following cases:
\begin{enumerate}
\item[1)] for $\lambda\in  ( -\frac{\pi}{2} \sum_{j=1}^3 c_j, -\sum_{j=1}^3 c_j)$ and $c_i>0$,
\item[2)] for $\lambda\in  ( -\infty, -\frac{\pi}{2} \sum_{j=1}^3 c_j]$ and $c_1=c_2=c_3$.
\end{enumerate}
\end{theorem}

\begin{theorem}\label{0antimain} 
Let $c_1=c_2=c_3$, $\lambda\in  (-\infty, +\infty)$,  and the smooth family of stationary anti-kink/kink profiles $\lambda\to \Phi_{\lambda,\delta'}$ determined above. Then $\Phi_{\lambda, \delta'}$ is spectrally unstable for the sine-Gordon model \eqref{sg2} on a tricrystal junction. 
\end{theorem}

In our stability analysis below, the family of linearized operators around the stationary profiles plays a fundamental role. These operators are characterized by the following formal self-adjoint diagonal matrix operators,
\begin{equation}\label{trav23}
\mathcal{W} \bold{v}=\Big (\Big(-c_j^2\frac{d^2}{dx^2}v_j + \cos (\varphi_j)v_j
\Big)\delta_{j,k} \Big ),\quad \l1\leqq j, k\leqq 3,\;\;\bold{v}= (v_j)_{j=1}^3,
\end{equation}
where $\delta_{j,k}$ denotes the Kronecker symbol, and defined on domains with $\delta'$-type interaction at the vertex $\nu = 0$, 
\begin{equation}\label{2trav23}
D(\mathcal{W})= \Big \{\bold{v}=(v_j)_{j=1}^3\in H^2(\mathcal{Y}):
c_1v'_1(0+)=c_2v'_2(0+)=c_3v_3(0+),\;\; \sum\limits_{j=1}^3 c_jv_j(0+)=\lambda c_1v'_1(0+) \Big \},
\end{equation}
with $\lambda \in \R$.
 It is to be observed that the particular family \eqref{tricry1u} of kink-profile stationary solutions under consideration is such that $ (\varphi_j)_{j=1}^3 \in D(\mathcal{W})$ in view that they satisfy the boundary conditions \eqref{bcIIfam}.

In section \S \ref{seccriterium} we establish a general instability criterion for static solutions for the sine-Gordon model \eqref{sg2} on a $\mathcal{Y}$-junction. The reader can find this result in Theorem \ref{crit} below. It essentially provides sufficient conditions on the flow of the semigroup generated by the linearization around the stationary solutions, for the existence of a pair of positive/negative real eigenvalues of this linearization, which determine the Morse index for the associated self-adjoint operator $(\mathcal{W}, D(\mathcal{W})$. The proof of Theorems \ref{trimain} and \ref{0antimain} follow as an application of Theorem \ref{crit}.

The structure of the paper is the following:  in section \S \ref{seccriterium}, we review the general instability criterion for stationary solutions for the sine-Gordon model \eqref{sg2} on a $\mathcal{Y}$-junction developed in \cite{AnPl-delta} (see Theorem \ref{crit} below; see also \cite{AngCav}). It is to be observed that this instability criterion is very versatile, as it applies to any type of stationary solutions (such as anti-kinks or breathers, for example) and for different interactions at the vertex, such as both the $\delta$- and $\delta'$-types. The subsection 3.1 is devoted to develop the instability theory of kink-profiles and we show Theorem \ref{trimain}. A special space is defined in order to analyze the Cauchy problem in $H^1(\mathcal Y)\times L^2(\mathcal Y)$. Subsection 3.2 is devoted to the proof of  Theorem \ref{0antimain}. By convenience of the reader and by the sake of completeness we establish in the Appendix some results of the extension theory of symmetric operators used in the body of the manuscript.

\subsection*{On notation}

For any $-\infty\leq a<b\leq\infty$, we denote by $L^2(a,b)$ the Hilbert space equipped with the inner product $
(u,v)=\int_a^b u(x)\overline{v(x)}dx$.
By $H^n(a,b)$  we denote the classical  Sobolev spaces on $(a,b)\subseteq \mathbb R$ with the usual norm.   We denote by $\mathcal{Y}$ the junction of type II parametrized by the edges $E_j = (0,\infty)$, $j =1,2,3$,  attached to a common vertex $\nu=0$. On the graph $\mathcal{Y}$ we define the classical $L^p$-spaces 
  \begin{equation*}
  L^p(\mathcal{Y})=L^p(0, +\infty)  \oplus L^p(0, +\infty) \oplus L^p(0, +\infty), \quad \,p>1,
  \end{equation*}   
 and  Sobolev-spaces $H^m(\mathcal{Y})=H^m(0, +\infty) \oplus H^m(0, +\infty)  \oplus H^m(0, +\infty)$  
with the natural norms. Also, for $\mathbf{u}= (u_j)_{j=1}^3$, $\mathbf{v}= (v_j)_{j=1}^3 \in L^2(\mathcal{Y})$, the inner product is defined by
$$
\langle \mathbf{u}, \mathbf{v}\rangle=  \sum_{j=1}^3 \int_0^{\infty}  u_j(x) \overline{v_j(x)} \, dx
$$
Depending on the context we will use the following notations for different objects. By $\|\cdot \|$ we denote  the norm in $L^2(\mathbb{R})$ or in $L^2(\mathcal{Y})$. By $\| \cdot\| _p$ we denote  the norm in $L^p(\mathbb{R})$ or in $L^p(\mathcal{Y})$.   Finally, if $A$ is a closed, densely defined symmetric operator in a Hilbert space $H$ then its domain is denoted by $D(A)$, the deficiency indices of $A$ are denoted by  $n_\pm(A):=\dim  \ker (A^*\mp iI)$, where $A^*$ is the adjoint operator of $A$, and the number of negative eigenvalues counting multiplicities (or Morse index) of $A$ is denoted by  $n(A)$.

 \section{Preliminaries: Linear instability criterion for sine-Gordon model on a tricrystal junction}
\label{seccriterium}

We start our study by recasting the equations in \eqref{sg2} in the vectorial form 
  \begin{equation}\label{stat1}
  \bold w_t=JE\bold w +F(\bold w)
 \end{equation} 
  where $\bold w=(u, v)^\top$, with $u=(u_1, u_2, u_3)^\top$, $v=(v_1, v_2, v_3)^\top$, $u_j, v_j: E_j \to \mathbb R$, $1 \leqq j \leqq 3$,
 \begin{equation}\label{stat2} 
J=\left(\begin{array}{cc} 0 & I_3 \\ -I_3  & 0 \end{array}\right),\quad E=\left(\begin{array}{cc} \mathcal T& 0 \\0 & I_3\end{array}\right), \quad 
F(\bold w)=\left(\begin{array}{cc}  0\\  0   \\  0   \\  -\sin (u_1)   \\  -\sin (u_2)   \\  -\sin (u_3)  \end{array}\right),
\end{equation} 
 and where $I_3$ denotes the identity matrix of order $3$ and $ \mathcal T$ is the diagonal-matrix linear operator
 \begin{equation} 
 \label{Lapla}
 \mathcal{T}=\Big (\Big(-c_j^2\frac{d^2}{dx^2}
\Big)\delta_{j,k} \Big ),\quad\l1\leqq j, k\leqq 3.
   \end{equation}

For the $\mathcal{Y}$-junction being a tricrystal junction, we will use the $\delta'$-interaction domain for $\mathcal T_\lambda\equiv\mathcal T$ given by \eqref{2trav23}, namely, $D(\mathcal{T}_{\lambda}) =D(\mathcal{W})$:
\begin{equation}
\label{tricry2}
D(\mathcal{T}_{\lambda}) :=  \{(v_j)_{j=1}^3\in H^2(\mathcal Y):
c_1v'_1(0+)=c_2v'_2(0+)=c_3v'_3(0+),\;\; \sum\limits_{j=1}^3 c_jv_j(0+)=\lambda c_1v'_1(0+) \},
\end{equation}
with $\lambda\in \mathbb R$ (see Proposition \ref{T}). 

In the sequel, we review the linear instability criterion of stationary solutions for the sine-Gordon model \eqref{sg2} on a  $\mathcal{Y}$-junction developed in \cite{AnPl-delta}. Although the stability analysis in \cite{AnPl-delta} pertains to interactions of $\delta$-type at the vertex, it is important to note that the criterion proved in that reference also applies to any type of stationary solutions independently of the boundary conditions under consideration and, therefore, it can be used to study the present configurations with boundary rules at the vertex of $\delta'$-interaction type, or even to other types of stationary solutions to the sine-Gordon equation such as anti-kinks and breathers, for instance. In addition, the criterion applies to both the $\mathcal Y$-junction of type I (see Figure \ref{figYtipoI}) and of type II (see Figure \ref{figYtipoII}). 

Let $\mathcal{Y}$ be a tricrystal junction. Let us suppose that $JE$ on a domain $D(JE)\subset H^1(\mathcal Y)\times L^2(\mathcal Y)$ is the infinitesimal generator of a $C_0$-semigroup on $H^1(\mathcal Y)\times L^2(\mathcal Y)$ and that there exists an stationary solution $\Upsilon=(\zeta_1, \zeta_2, \zeta_3,0,0,0)\in D(JE)$. Thus, every component $\zeta_j$ satisfies the equation
\begin{equation}\label{statio}
-c^2_j \zeta''_j + \sin (\zeta_j)=0,\quad j=1,2,3.
\end{equation}
 Now, we suppose  that $\bold w$ satisfies formally  equality in \eqref{stat1} and we define
  \begin{equation}\label{stat3}
 \bold v \equiv \bold w-\Upsilon,
  \end{equation}
  then, from \eqref{statio} we obtain the following  linearized system for \eqref{stat1} around $\Phi$,
   \begin{equation}\label{stat4}
  \bold v_t=J\mathcal E\bold v,
 \end{equation} 
  with $\mathcal E$ being the $6\times 6$ diagonal-matrix $\mathcal E=\left(\begin{array}{cc} \mathcal L & 0 \\0 & I_3\end{array}\right)$, and 
 \begin{equation}\label{stat5}
\mathcal{L} =\Big (\Big(-c_j^2\frac{d^2}{dx^2}+ \cos (\zeta_j)
\Big)\delta_{j,k} \Big ),\qquad 1\leqq j, k\leqq 3.
\end{equation}
We point out the equality $J\mathcal E=J E+\mathcal{Z}$, with 
\begin{equation}\label{Z}
\mathcal{Z} = \left(\begin{array}{cc} 0& 0 \\ \big( - \cos(\zeta_j) \, \delta_{j,k} \big)& 0\end{array}\right)
\end{equation}
 being a \emph{bounded} operator on $H^1(\mathcal Y)\times L^2(\mathcal Y)$. This  implies that $J\mathcal E$ also generates a  $C_0$-semigroup on  $H^1(\mathcal Y)\times L^2(\mathcal Y)$ (see Pazy \cite{Pa}).
  
The linear instability criterion provides sufficient conditions for the trivial solution $\bold v \equiv 0$ to be unstable by the linear flow of \eqref{stat4}. More precisely, it underlies the existence of a {\it growing mode solution} to \eqref{stat4} of the form $ \bold v=e^{\lambda t} \Psi$ and $\RE \lambda >0$. To find it, one needs to solve the formal system 
 \begin{equation}\label{stat6}
 J\mathcal E \Psi=\lambda \Psi,
\end{equation}
with $\Psi\in D(J\mathcal E)$. If we denote by $\sigma(J\mathcal E)= \sigma_{\mathrm{pt}}(J\mathcal E)\cup \sigma_{\mathrm{ess}}(J\mathcal E)$ the spectrum  of $J\mathcal E$ (namely, $\lambda \in  \sigma_{\mathrm{pt}}(J\mathcal E)$ if $\lambda$ is isolated and with finite multiplicity) then we have the following
\begin{definition}
The stationary vector solution $\Upsilon \in D(\mathcal E)$    is said to be \textit{spectrally stable} for model sine-Gordon \eqref{stat1} if the spectrum of $J\mathcal E$, $\sigma(J\mathcal \mathcal E)$, satisfies $\sigma(J\mathcal E)\subset i\mathbb{R}.$
Otherwise, the stationary solution $\Upsilon\in D(\mathcal E)$   is said to be \textit{spectrally unstable}.
\end{definition}
\begin{remark}
It is well-known that $ \sigma_{\mathrm{pt}}(J\mathcal E)$ is symmetric with respect to both the real and imaginary axes and $ \sigma_{\mathrm{ess}}(J\mathcal E)\subset i\mathbb{R}$ under the assumption that $J$ is skew-symmetric and that $\mathcal E$ is self-adjoint (by supposing, for instance, Assumption $(S_3)$ below for $\mathcal L$; see \cite[Lemma 5.6 and Theorem 5.8]{GrilSha90}). These cases on $J$ and  $\mathcal E$ are considered in the theory. Hence, it is equivalent to say that $\Upsilon\in D(J\mathcal E)$ is  \textit{spectrally stable} if $ \sigma_{\mathrm{pt}}(J\mathcal E)\subset i\mathbb{R}$, and it is spectrally unstable if $ \sigma_{\mathrm{pt}}(J\mathcal E)$ contains point $\lambda$ with  $\RE \lambda>0.$
\end{remark}

 It is widely known  that the spectral instability of a specific traveling wave solution of an evolution type model is   a key prerequisite to show their nonlinear instability property (see \cite{GrilSha90, Lopes, ShaStr00} and the references therein). Thus we have the following definition.

 \begin{definition}\label{nonstab}
The stationary vector solution $\Upsilon \in D(\mathcal E)$    is said to be \textit{nonlinearly unstable} in $X\equiv H^1(\mathcal Y)\times L^2(\mathcal Y)$-norm for model sine-Gordon \eqref{stat1} if there is $\epsilon>0$ such that for every $\delta>0$ there exist an initial data $\bold w_0$ with $\|\Upsilon -\bold w_0\|_X<\delta$ and an instant $t_0=t_0(\bold w_0)$, such that $\|\bold w(t_0)-\Upsilon \|_X>\epsilon$, where $\bold w=\bold w(t)$ is the solution of the sine-Gordon model with initial data $\bold w(0)=\bold w_0$.
\end{definition}
 
 From \eqref{stat6}, the eigenvalue problem to solve is now reduced to
  \begin{equation}\label{stat11}    
 J\mathcal E\Psi =\lambda \Psi, \quad \RE \lambda >0,\;\;\Psi\in D(\mathcal E).
 \end{equation}  
 Next, we establish our theoretical framework and assumptions for obtaining a nontrivial solution to  problem in \eqref{stat11}:
 \begin{enumerate}
 \item[($S_1$)]  $J\mathcal E$ is the generator of a $C_0$-semigroup $\{S(t)\}_{t\geqq 0}$.  
 \item[($S_2$)] Let $\mathcal L$ be the matrix-operator in \eqref{stat5}  defined on a domain $D(\mathcal L)\subset L^2(\mathcal Y)$ on which $\mathcal L$ is self-adjoint.
 \item[($S_3$)] Suppose $\mathcal L:D(\mathcal L)\to L^2(\mathcal Y)$ is  invertible  with Morse index $n(\mathcal L)=1$ and such that $\sigma(\mathcal L)=\{\lambda_0\}\cup J_0$ with $J_0\subset [r_0, +\infty)$, for $r_0>0$, and $\lambda_0<0$,
 \end{enumerate} 
 
The criterion for linear instability reads precisely as follows (cf. \cite{AnPl-delta}).
\begin{theorem}[linear instability criterion]
\label{crit}
Suppose the assumptions $(S_1)$ - $(S_3)$ hold.  Then the operator $J\mathcal E$ has a real positive and a real negative eigenvalue. 
\end{theorem}
\begin{proof}
See \cite{AngCav, AnPl-delta}.
\end{proof}

\section{Instability of stationary solutions for the sine-Gordon equation with $\delta'$-interaction on a tricrystal junction}
\label{secinsI}

In this section we study the stability of stationary solutions determined by a $\delta'$-interaction type at the vertex  $\nu=0$ of a tricrystal junction. First we study the kink-profile type in \eqref{tricry1u}-\eqref{bcinfty}, thus the local well-posedness problem associated to \eqref{sg2} with a $\delta'$-interaction. Next,  we apply the linear instability criterion (Theorem \ref{crit}) to prove 
that the family of stationary solutions \eqref{tricry1u} are linearly (and nonlinearly) unstable (Theorem \ref{trimain} above). Our second focus goes to the study of the anti-kink-profile type in \eqref{antikink} and similarly as in the former profile case we establish the necessary ingredients for obtaining Theorem \ref{0antimain} above.

\subsection{Kink-profile's instability  on a tricrystal junction}

We start our stability study for the kink-profile type in \eqref{tricry1u}-\eqref{bcinfty} and so our first focus is dedicated to the Cauchy problem associated to the sine-Gordon model in \eqref{stat1}. As this study is not completely  standard in the case of metric graphs we provide the new ingredients that arise.

\subsubsection{Cauchy Problem in $H^1(\mathcal Y)\times L^2(\mathcal Y)$}

In this subsection we establish the local well-posedness in $H^1(\mathcal Y)\times L^2(\mathcal Y)$ of the  sine-Gordon equation on a tricrystal junction (section \S 2 in \cite{AnPl-delta}). We start with the following result from the extension theory. The proof follows the same strategy as in Proposition A.6 and Theorem 3.1 in Angulo and Plaza \cite{AnPl-deltaprime} (see Proposition \ref{11} in the Appendix) and we omit it.

\begin{proposition}\label{T}
Consider the closed symmetric operator $(\mathcal T, D( \mathcal T))$ densely defined on $L^2(\mathcal Y)$, with $\mathcal Y$ being a tricrystal junction, by
\begin{equation}\label{tricr2}
\begin{split}
&\mathcal{T}=\Big (\Big(-c_j^2\frac{d^2}{dx^2}\Big)\delta_{j,k} \Big ),\;\;1\leqq j, k\leqq 3,\\
&D(\mathcal{T})= \Big \{(v_j)_{j=1}^3\in H^2(\mathcal{Y}):
c_1v_1'(0+)=c_2v'_2(0+)=c_3v_3'(0+)=0,\;\; \sum_{j=1}^3 c_jv_j(0+)=0 \Big \}.
\end{split}
\end{equation}
Here $c_j\neq 0$, $1\leqq j\leqq 3$, and $\delta_{j,k}$ is the Kronecker symbol. Then, the deficiency indices are $n_{\pm}( \mathcal T)=1$. Therefore, we have that all the self-adjoint extensions of $( \mathcal T, D( \mathcal T))$ are given by the one-parameter  family $(\mathcal T_\lambda, D(\mathcal T_\lambda))$,  $\lambda\in \mathbb R$, with $\mathcal T_\lambda\equiv \mathcal T$ and $D(\mathcal T_\lambda)$  defined by
 \begin{equation}\label{tricr3}
D(\mathcal T_\lambda)= \{(u_j)_{j=1}^3\in H^2(\mathcal{Y}):
c_1u'_1(0+)=c_2u'_2(0+)=c_3u'_3(0+),\;\; \sum_{j=1}^3 c_ju_j(0+)=\lambda c_1u'_1(0+) \}.
\end{equation}

Moreover, the spectrum of the family of self-adjoint operators $(\mathcal T_\lambda, D(\mathcal T_\lambda))$ satisfies $\sigma_{\mathrm{ess}}(\mathcal T_\lambda)=[0, +\infty)$ for every $\lambda\neq 0$. For $\lambda<0$,  $\mathcal T_\lambda$ has precisely one negative simple eigenvalue. If $\lambda>0$ then $\mathcal T_\lambda$ has no eigenvalues (see \cite{AngGol18b}). 
\end{proposition}

\begin{theorem}\label{well3} Let $\mathcal Y=(0,+\infty)\cup (0,+\infty)\cup(0,+\infty)$.
	For any $\Psi \in H^1(\mathcal{Y})\times L^2(\mathcal Y) $
	there exists $T > 0$ such that the sine-Gordon
	equation \eqref{stat1} has a unique solution $\mathbf w \in C
	([0,T]; H^1(\mathcal{Y})\times  L^2(\mathcal Y) )$ satisfying  $\mathbf
	w(0)=\Psi$. For
	each $T_0\in (0, T)$ the mapping data-solution
	\begin{equation}\label{mapping}
	\Psi\in H^1(\mathcal{Y})\times  L^2(\mathcal Y) \to \mathbf w \in C ([0,T_0];
	H^1(\mathcal{Y})\times  L^2(\mathcal Y) ),
	\end{equation}
	 is at least of class $C^2$. 
	\end{theorem}

\begin{proof} By applying the same strategy as in \cite{AnPl-deltaprime} (Theorems 3.2 and 3.3), we have the following observations:
\begin{enumerate}
\item[1)] Consider the linear operators $J$ and $E$ defined in \eqref{stat2} with $(\mathcal T_\lambda, D(\mathcal T_\lambda))$ defined in Proposition \ref{T}. Then, $\mathcal A\equiv JE$ with $D(\mathcal A)= D(\mathcal T_\lambda)\times H^1(\mathcal Y)$ is the infinitesimal  generator of a $C_0$-semigroup  on $H^1(\mathcal Y)\times L^2(\mathcal Y)$.
\item[2)]  By using the contraction mapping principle, we obtain the local well-posedness result for the sine-Gordon equation \eqref{stat1} on $H^1(\mathcal{Y})\times L^2(\mathcal Y)$. The $C^2$-property of the mapping data-solution follows from the implicit function theorem.
\end{enumerate}
\end{proof}

 We are ready to draw conclusions about the linearized operator around the family of stationary solutions of the form \eqref{tricry1u} on a tricrystal junction required by the assumption in section \S \ref{seccriterium}.
\begin{proposition}
\label{coro1} 
Consider the operator $(\mathcal W_\lambda, D(\mathcal W_\lambda ))$ determined by $\mathcal{W}_\lambda \equiv \mathcal{W}$ defined in \eqref{trav23}, on the domain $D(\mathcal{W}_\lambda) = D(\mathcal{T}_{\lambda})$ defined in \eqref{tricr3}. Let $\mathcal E$  be the following diagonal-matrix 
$$
\mathcal E=\left(\begin{array}{cc} \mathcal W_\lambda & 0 \\0 & I_3\end{array}\right).
$$
  Then $J\mathcal E$ is the generator of a $C_0$-semigroup on $H^1(\mathcal{Y})\times L^2(\mathcal Y)$ with $D(J\mathcal E)=D(\mathcal W_\lambda )\times H^1(\mathcal Y)$. This implies, in turn, that assumption $(S_1)$ (see section \S \ref{seccriterium}) is satisfied.
\end{proposition}

\begin{proof} From the relation $J\mathcal E=JE + \mathcal Z$ (see \eqref{Z} for $\zeta_j=\varphi_j$), standard semigroup theory and item 1) in the proof of Theorem \ref{well3} imply the result.
\end{proof}

\subsubsection{Kink-profile for the sine-Gordon equation on a tricrystal-junction}
\label{secprofilesII}

We will consider  stationary solutions $(\varphi_j)_{j=1}^3$ for the  sine-Gordon equation on a tricrystal-junction  of the form \eqref{tricry1u} satisfying the $\delta'$-interactions at the vertex given by \eqref{bcIIfam},  this means that $(\varphi_j)_{j=1}^3\in D(\mathcal{T}_{\lambda})$.
%
%
%
Here we shall consider the {\it full continuity case} at  the vertex, under which
\[
\varphi_1(0+)=\varphi_2(0+)= \varphi_3(0+).
\] 
Hence $b_1/c_1=b_2/c_2=b_3/c_3$ and, moreover, the conditions $c_1 \varphi'_1(0+)=c_2\varphi'_2(0+)=c_3 \varphi'_3(0+)$ hold. The Kirchhoff type condition in \eqref{bcIIfam} implies, for $y=e^{b_1/c_1}$, the relation 
\begin{equation}\label{arct}
\frac{1+ y^2}{y}\arctan(y) \sum_{j=1}^3 c_j=-\lambda.
\end{equation}
Thus from the strictly-increasing property of the positive mapping $y\mapsto \frac{1+ y^2}{y}\arctan(y)$, $y>0$, we obtain from \eqref{arct} that $\lambda\in (-\infty, -\sum_{j=1}^3 c_j)$ and  the existence of a smooth mapping $\lambda\mapsto b_1(\lambda)$ satisfying \eqref{arct}. Moreover, $\lambda\in (-\infty, -\sum_{j=1}^3c_j)\mapsto \Pi_{\lambda, \delta'}=(\varphi_{1,b_1(\lambda)}, \varphi_{2, b_2(\lambda)}, \varphi_{3, b_3(\lambda)},0,0,0)$ represents a real-analytic family of static profiles for the sine-Gordon equation on a tricrystal-junction. Thus, we have:
\begin{enumerate}
\item[1)] for  $\lambda\in (-\infty, -\frac{\pi}{2} \sum_{j=1}^3 c_j)$ we obtain that $b_i>0$ and $\varphi_i''(b_i) =0$, for every $i$. Moreover, $\varphi_i(0+)\in (\eta, 2\pi)$, $i=1,2, 3$, with $\eta=4 \arctan\big(e^{b_1/c_2}\big)>\pi$. Thus, the profile $(\varphi_1, \varphi_2,\varphi_3)$ looks like that presented in Figure \ref{figBumpYI} below (bump-type profile);
\item[2)] for $\lambda \in (-\frac{\pi}{2} \sum_{j=1}^3 c_j, -\sum_{j=1}^3 c_j)$ we obtain $b_i<0$, $ \varphi''_i>0$, $\varphi_i\in (0, \pi)$ for every $i$. Thus, the profile $(\varphi_1, \varphi_2,\varphi_3)$ is of tail-type as that of Figure \ref{figTailYII} below; 
\item[3)] for $\lambda=-\frac{\pi}{2} \sum_{j=1}^3 c_j$ we obtain $b_i=0$, $\varphi_i(0)=\pi$ and $\varphi_i''(0) =0$, for every $i$. Moreover, $\varphi_i''(x) >0$ for $x>0$. Thus, the profile $(\varphi_1, \varphi_2,\varphi_3)$ is similar to that of Figure \ref{figZeroWYII} below.
\end{enumerate}

\begin{figure}[th]
\begin{center}
\subfigure[$\lambda\in (-\infty,-\tfrac{3\pi}{2})$]{\label{figBumpYI}\includegraphics[scale=.44, clip=true]{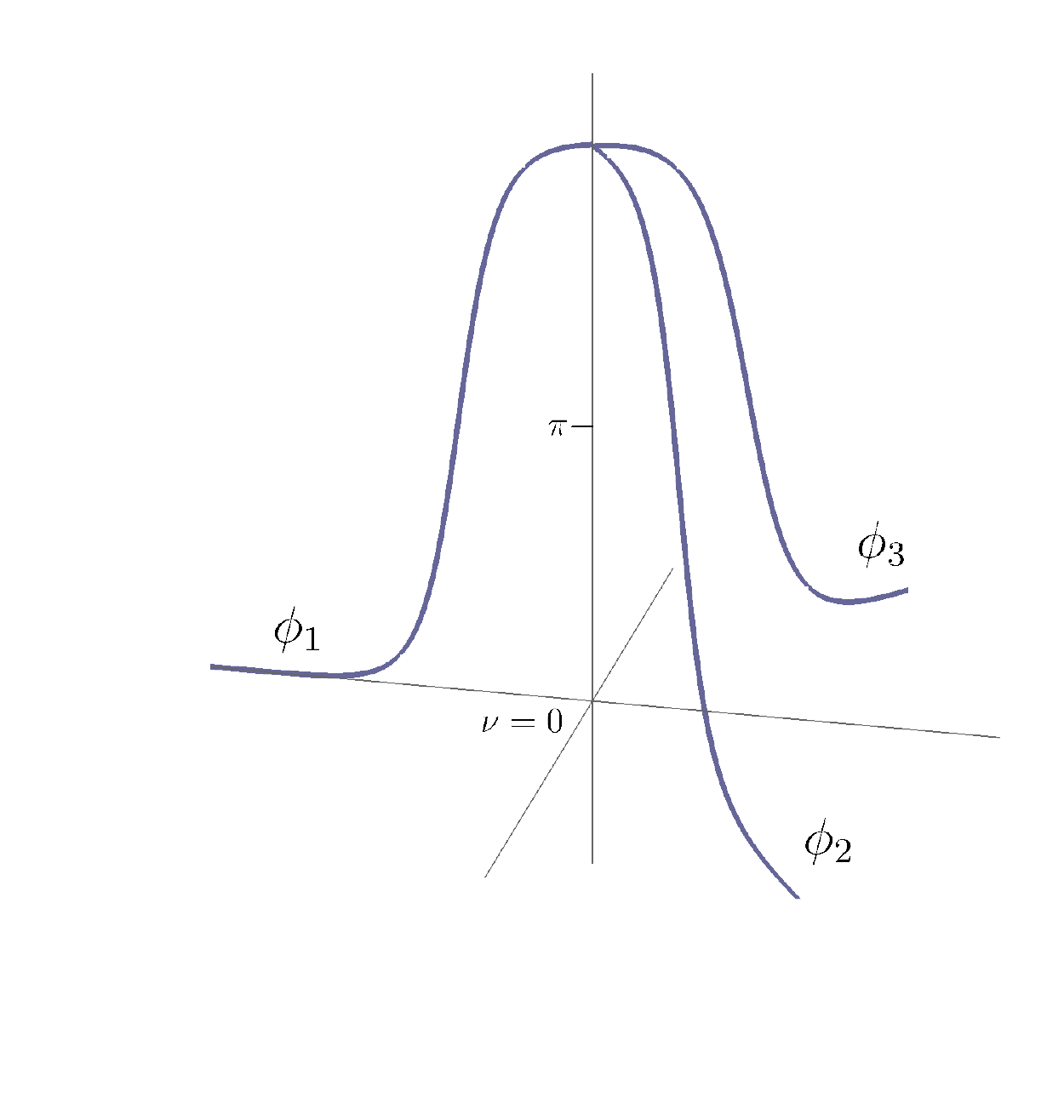}}
\subfigure[$\lambda\in (-\tfrac{3\pi}{2},-3)$]{\label{figTailYII}\includegraphics[scale=.44, clip=true]{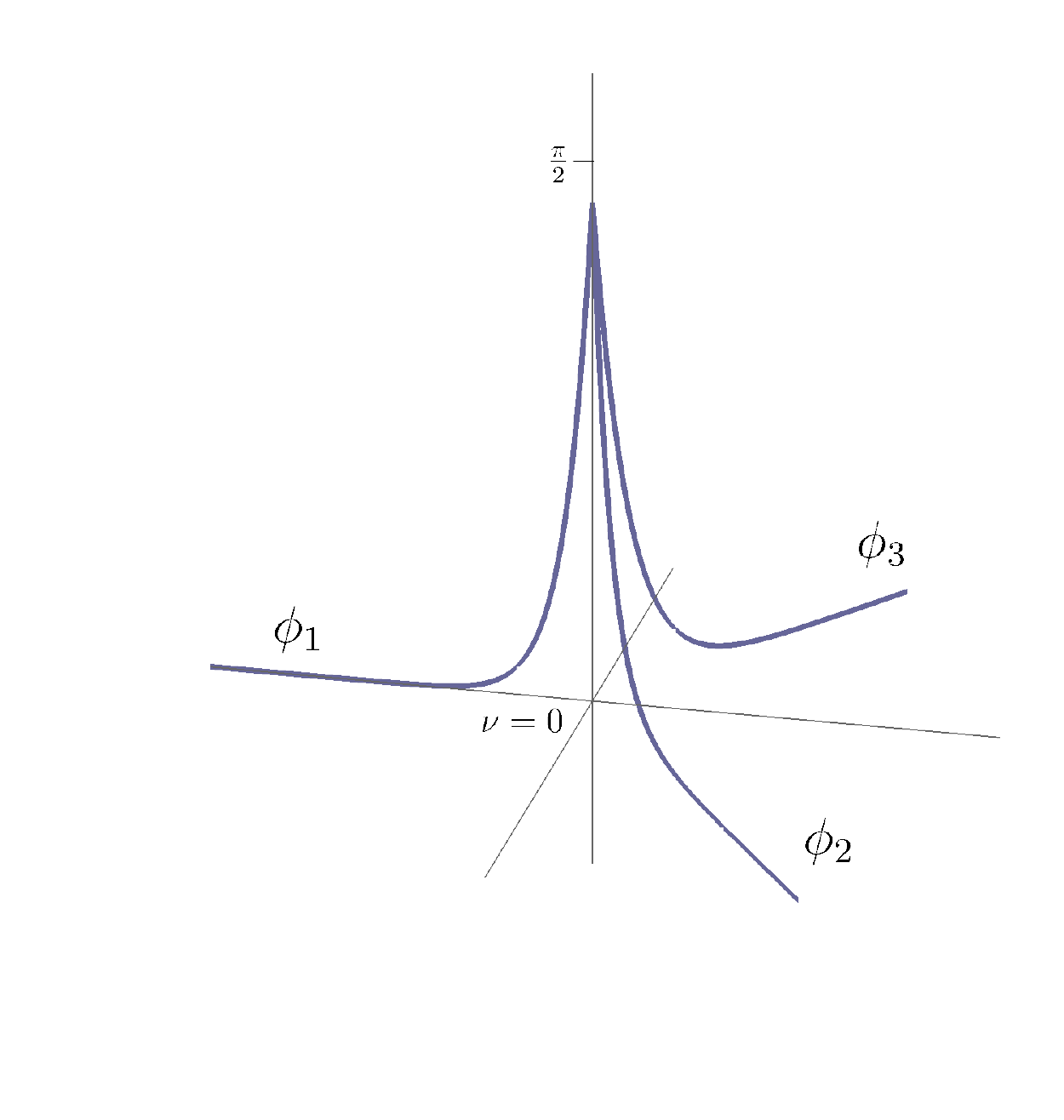}}
\subfigure[$\lambda = -\tfrac{3\pi}{2}$]{\label{figZeroWYII}\includegraphics[scale=.44, clip=true]{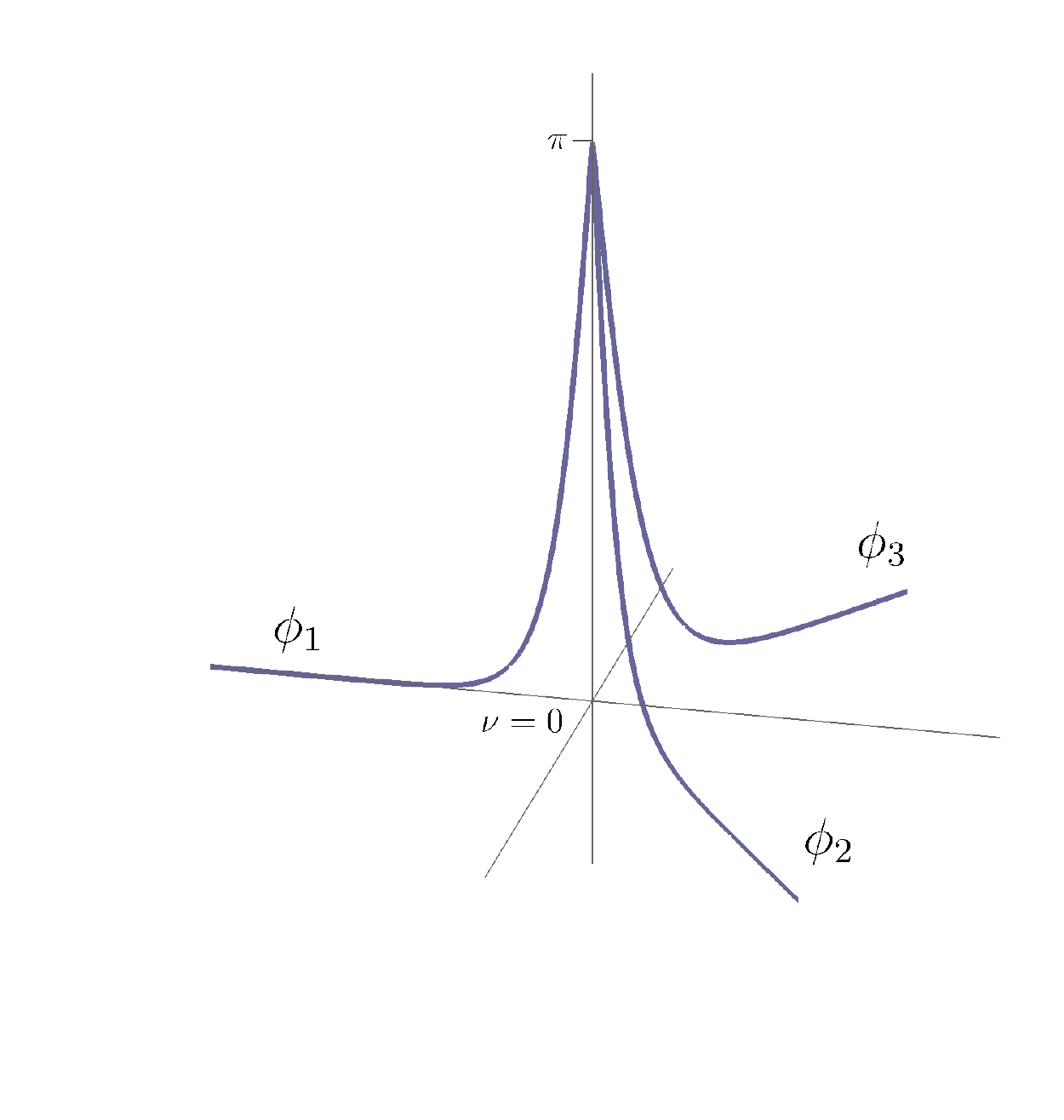}}
\end{center}
\caption{\small{Plots of stationary solutions \eqref{tricry1u} in the case where $c_j = 1$ for all $j=1,2,3$, for different values of $\lambda \in (\infty, -\sum_j c_j) = (-\infty, -3)$. Panel (a) shows the stationary profile solutions (``bump'' configuration) for the case $\lambda \in (-\infty,-\tfrac{3\pi}{2})$. Panel (b) shows the profiles of\/ ``tail'' type for the case $\lambda \in (-\tfrac{3\pi}{2},-3)$. Panel (c) shows the ``smooth'' profile solutions when $\lambda = -\tfrac{3\pi}{2}$ (color online).}}\label{figYII}
\end{figure}

The stability result for the stationary profiles, $\Pi_{\lambda,\delta'}=(\varphi_1,\varphi_2,\varphi_3, 0,0,0)$, with $\varphi_j=\varphi_{j, b_j(\lambda)}$ defined in \eqref{tricry1u} and \eqref{arct} in the continuous case is that established in Theorem \ref{trimain}.  We leave for a  possible future study the stability analysis of other kink-type profiles, such as those non-continuous at the vertex.

The proof of Theorem \ref{trimain} is a consequence of Theorem \ref{crit}. Thus we only need to verify assumption $(S_3)$ associated to the family of self-adjoint operators  $\mathcal W_\lambda$ in \eqref{trav23} (see Proposition \ref{T}),
with the domain $D(\mathcal{W}_\lambda)=D(\mathcal{T}_{\lambda})$ in \eqref{tricr3} and the kink-profiles $\varphi_j$ defined in \eqref{tricry1u}. 

\subsubsection{Spectral study for $(\mathcal W_\lambda, D(\mathcal{W}_\lambda))$ on a tricrystal-junction}

In this subsection, the  spectral behavior for $\mathcal W_\lambda$ on $D(\mathcal{W}_\lambda)$ will be study with the focus of verifying assumption $(S_3)$ in  section 2.

\begin{proposition}\label{1crystal}
Let $\lambda\in  (-\infty, -\sum_{j=1}^3 c_j)$, $c_j>0$. Then for $\lambda\neq  -\frac{\pi}{2} \sum_{j=1}^3 c_j$ we have $\ker( \mathcal{W}_\lambda )=\{0\}$. For $\lambda_0= -\frac{\pi}{2} \sum_{j=1}^3c_j$, $\ker( \mathcal{W}_{\lambda_0} )=2$. Moreover, $\sigma_{\mathrm{ess}}( \mathcal{W}_{\lambda})=[1,+\infty)$.
\end{proposition}

\begin{proof}  We consider $\bold{u}=(u_1, u_2, u_3)\in D(\mathcal{W}_\lambda)$ and $\mathcal{W}_\lambda \bold{u}=\bold{0}$. Then, from Sturm-Liouville theory on half-lines (see \cite{BeSh}, Chapter 2, Theorem 3.3), $u_j(x)=\alpha_j\varphi'_j(x)$, $x>0$, $j=1,2,3$. Thus,  for $\lambda\neq  -\frac{\pi}{2} \sum_{j=1}^3 c_j$ we obtain $\alpha_1/c_1=\alpha_2/c_2=\alpha_3/c_3$. Next, from the jump conditions for $\bold{u}$ and $\Psi_{\lambda,\delta'}=(\varphi_j)$,  we obtain
\begin{equation}\label{tricry5}
\alpha_1\varphi'_1(0)\sum_{j=1}^3 c_j=\alpha_1\lambda c_1 \varphi''_1(0).
\end{equation}
Next, suppose $\alpha_1\neq 0$. Since   $\varphi'_1(0)<0$ and for $\lambda\in  (-\infty, -\frac{\pi}{2}\sum_{j=1}^3 c_j)$ we have $\varphi''_1(0)<0$,  relation in \eqref{tricry5} implies a contradiction because of $c_j>0$. Now, considering $\lambda\in  ( -\frac{\pi}{2}\sum_{j=1}^3 c_j, -\sum_{j=1}^3 c_j)$ and from the specific values of $\varphi'_1(0), \varphi''_1(0)$ we obtain from \eqref{tricry5} again a contradiction. Thus, from the two cases above we need to have $\alpha_1=\alpha_2=\alpha_3=0$.  For $\lambda_0= -\frac{\pi}{2} \sum_{j=1}^3c_j$ we recall that $\varphi''_j(0)=0$ for every $j$. Hence, from the jump-condition for $\bold u$ follows $\sum_{j=1}^3 \alpha_j=0$. Then $\Psi_1=(\varphi'_1, -\varphi'_2, 0)$ and $\Psi_2=(0,\varphi'_2, -\varphi'_3)$ belong to $D(\mathcal W_{\lambda_0})$ and $\Span\{\Psi_1, \Psi_2\}=\ker(\mathcal W_{\lambda_0})$.

The statement $\sigma_{\mathrm{ess}}(\mathcal{W}_\lambda)=[1,+\infty)$ is an immediate consequence of Weyl's Theorem because of $\lim_{x\to +\infty} \cos(\varphi_1(x))=1=\lim_{x\to +\infty} \cos(\varphi_j(x))$ (see \cite{RS4}). This finishes the proof.
\end{proof}

\begin{proposition}\label{2crystal}
Let $\lambda\in [-\frac{\pi}{2} \sum_{j=1}^3 c_j, -\sum_{j=1}^3 c_j)$. Then $n( \mathcal{W}_\lambda )=1$.
\end{proposition}

\begin{proof} We will use  the extension theory approach, which is based on the fact that the family $(\mathcal W_\lambda, D(\mathcal W_\lambda))$ represents all the self-adjoint extensions  of the closed symmetric operator $(\widetilde{T}, D(\widetilde{T}))$ with $D( \widetilde{T})\equiv D( \mathcal T)$ defined in \eqref{tricr2}, and 
\begin{equation}\label{tricry6}
\widetilde{T}=\Big (\Big(-c_j^2\frac{d^2}{dx^2}+\cos(\varphi_j)
\Big)\delta_{j,k} \Big ),\;\l1\leqq j, k\leqq 3,\\
\end{equation}
with $n_{\pm}(\widetilde{T})=1$. Next, we show that  $\widetilde{T}\geqq 0$. If we denote $L_j=-c_j^2\frac{d^2}{dx^2}+\cos(\varphi_j)$ then  we obtain 
\begin{equation}
\label{spec7}
L_j\psi=-\frac{1}{\varphi'_j} \frac{d}{dx}\Big[c_j^2  (\varphi'_j)^2 \frac{d}{dx}\Big(\frac{\psi}{\varphi'_j}\Big)\Big],
\end{equation}
for any $\psi$. It is to be observed that  $\varphi'_j\neq 0$. Then  we have for any $\Lambda=(\psi_j)\in D(\widetilde{T})$,
\begin{equation}\label{7spec}
\begin{split}
\langle \widetilde{T} \Lambda,\Lambda\rangle= A -\sum_{j=1}^3c_j^2\psi^2_j(0)\frac{\varphi''_j(0)}{\varphi'_j(0)}\equiv A+P,
\end{split}
\end{equation}
where $A\geqq 0$ represents the integral terms. Next we show that $P\geqq0$. Indeed, since $\varphi_j''(0)\geqq 0$ and  $\varphi_j'(0)<0$, for $j=1,2,3$, we obtain immediately $P\geqq 0$. Then, $ \widetilde{T}\geqq 0$.

Due to  Proposition \ref{semibounded} (see Appendix),  $n(\mathcal W_\lambda)\leqq 1$. Next, for $\Psi_{\lambda, \delta'}=(\varphi_1, \varphi_2, \varphi_3)\in D(\mathcal W_\lambda)$ 
 we obtain
\begin{equation}\label{negaqua}
\langle \mathcal W_\lambda \Psi_{\lambda, \delta'}, \Psi_{\lambda, \delta'}\rangle=\sum_{j=1}^3\int_0^{+\infty}[-\sin(\varphi_j)+\cos(\varphi_j)\phi_j]\varphi_jdx<0,
\end{equation}
because of $0<\varphi_j(x)\leqq \pi$ and $x\cos x\leqq \sin x$ for all $x\in [0, \pi]$.  Then from the minimax principle we  arrive at $n(\mathcal W_\lambda)=1$. This finishes the proof.
\end{proof}

\begin{proof}[Proof of Theorem \ref{trimain}]
Let $\lambda\in (-\frac{\pi}{2} \sum_{j=1}^3 c_j, -\sum_{j=1}^3 c_j)$. Then,  from Propositions \ref{1crystal} and \ref{2crystal} we have $\ker( \mathcal{W}_\lambda)=\{0\}$ and $n( \mathcal{W}_\lambda)=1$. Thus, from Theorem \ref{well3}, Corollary \ref{coro1} and  Theorem \ref{crit} follows the linear instability  property of the stationary profile $\Pi_{\lambda, \delta'}=(\varphi_1,\varphi_2, \varphi_3, 0,0,0)$. 

Next, we consider the closed subspace in $L^2(\mathcal Y)$, $\mathcal C_1=\{(u_j)_{j=1}^3\in L^2(\mathcal Y): u_1=u_2=u_3\}$,
 and the case $c_1=c_2=c_3$ (hence $b_1=b_2=b_3$). Then, we can show  that on 
$\mathcal B_1=\mathcal C_1 \cap D (\mathcal{W}_{\lambda})$  we obtain $\mathcal{W}_{\lambda}:\mathcal B_1\to \mathcal C_1$ is well-defined. Moreover, we note that for $\lambda_0=-\frac{\pi}{2} \sum_{j=1}^3c_j$ follows $\ker( \mathcal{W}_{\lambda_0}|_{\mathcal B_1})=\{0\}$ and $n( \mathcal{W}_{\lambda_0}|_{\mathcal B_1})=1$ (because $\Psi_{\lambda_0, \delta'}=(\varphi_1, \varphi_1, \varphi_1)\in \mathcal B_1$ and $\langle \mathcal W_{\lambda_0} \Psi_{\lambda_0, \delta'}, \Psi_{\lambda_0, \delta'}\rangle<0$).  Therefore, from Kato-Rellich Theorem,  analytic perturbation and a continuation argument (see  Angulo and Plaza  \cite{AnPl-delta} and/or Propositions \ref{antimain4}-\ref{Appanti} below), we can see that for all  $\lambda\in (-\infty, -\frac{\pi}{2} \sum_{j=1}^3c_j]$  the following relations hold: $\ker( \mathcal{W}_{\lambda}|_{\mathcal B_1})=\{0\}$ and $n( \mathcal{W}_{\lambda}|_{\mathcal B_1})=1$. Thus, the static profiles $\Pi_{\lambda, \delta'}$ are also  linearly unstable in this case.

Now, since the mapping data-solution for the sine-Gordon model on $H^1(\mathcal Y)\times L^2(\mathcal Y)$ is at least of class $C^2$ (indeed, it is smooth) by Theorem \ref{well3}, it follows that the linear instability property of $\Psi_{\lambda, \delta'}$   is in fact  of nonlinear type  in the $H^1(\mathcal Y)\times L^2(\mathcal Y)$-norm (see Henry {\it{et al.}} \cite{HPW82}, Angulo and Natali \cite{AngNat16}, and Angulo {\it{et al.}} \cite{ALN08}). This finishes the proof.
\end{proof}

\begin{remark}\label{fora}
For the case $\lambda\in (-\infty, -\frac{\pi}{2} \sum_{j=1}^3 c_j)$ in Proposition \ref{2crystal}, the formula for $P$ in \eqref{7spec}  satisfies $P<0$. Therefore, it is not clear whether the extension theory approach provides an estimate of the Morse-index of $\mathcal{W}_\lambda$; see also the related Remark 4.4 in \cite{AnPl-delta}.
\end{remark}

\subsection{Kink/anti-kink instability theory on a tricrystal junction}

In this subsection we study the existence and stability of kink/anti-kink profile in \eqref{antikink}.  Since these stationary profiles do not belong to the classical $H^2(\mathcal Y)$-energy space, we need to work with specific functional spaces suitable for our needs.

\subsubsection{The anti-kink/kink solutions  on a tricrystal junction}

In the sequel, we describe the profiles $(\varphi_j)_{j=1}^3$ in  \eqref{antikink} satisfying  the $\delta'$-condition in \eqref{bcIIfam}.  Thus, we obtain the following relations:
\begin{equation}\label{anti1}
\sech\big(\frac{a_1}{c_1}\big)=\sech\big(\frac{a_2}{c_2}\big)=\sech\big(\frac{a_3}{c_3}\big),
\end{equation}
and 
\begin{equation}\label{anti2}
c_1\arctan(e^{-a_1/c_1})+ c_2\arctan(e^{-a_2/c_2})+c_3\arctan(e^{-a_3/c_3})-\frac{(c_2+c_3)\pi}{2}=\frac{\lambda}{2} \sech \big( \frac{a_1}{c_1} \big).
\end{equation}

Here, we consider the specific class of anti-kink/kink profiles with the condition $\frac{a_1}{c_1}=\frac{a_2}{c_2}=\frac{a_3}{c_3}$. Thus, from \eqref{anti2}  and $y=e^{-a_1/c_1}$ we get the following equality
\begin{equation}\label{0anti3}
F(y)\equiv \frac{1+y^2}{y}\Big[\Big(\sum_{i=1}^3c_i\Big)\arctan(y)-\frac{(c_2+c_3)\pi}{2}\Big]=\lambda.
\end{equation}

Next, we show that the mapping $F$ is strictly increasing for $y>0$. Indeed, since for $y\in (0,1)$ we have that $\frac{1}{y}+ \frac{y^2-1}{y^2}\arctan(y)>0$, then $F'(y)>0$. Now, for $\theta=\sum_{i=1}^3c_i$, the relation
\begin{equation}\label{0anti4}
F'(y)=\theta\Big[\frac{1}{y}+ \frac{y^2-1}{y^2}\Big(\arctan(y)-\frac{\pi}{2}\Big)\Big]+\theta\Big[\frac{\pi}{2}-\frac{(c_2+c_3)\pi}{2 \sum_{i=1}^3c_i}\Big]\frac{y^2-1}{y^2}
\end{equation}
shows that  $F'(y)>0$ for $y\in [1, +\infty)$. Moreover, it is no difficult to see that $\lim_{y\to 0^+} F(y)=-\infty$, and $\lim_{y\to +\infty} F(y)=+\infty$. 

Then, from \eqref{0anti3} we have the following specific behavior of the $\lambda$-parameter:
\begin{enumerate}
\item[a)] for $a_1=0$, $\lambda=-\frac{\pi}{2}(c_2+c_3-c_1)$,
\item[b)] for $a_1>0$, $\lambda\in (-\infty, -\frac{\pi}{2}(c_2+c_3-c_1))$,
\item[c)] for $a_1<0$, $\lambda\in ( -\frac{\pi}{2}(c_2+c_3-c_1), +\infty)$.
\end{enumerate}
Moreover, from \eqref{0anti3} and the properties for $F$ we obtain the existence of a smooth shift-map (also real analytic) $\lambda\in (-\infty, +\infty)\mapsto a_1(\lambda)$ satisfying $F(e^{-a_1(\lambda)/c_1})=\lambda$. Thus,  the mapping 
$$
\lambda\in (-\infty, +\infty) \, \mapsto \Phi_{\lambda, \delta'}=(\varphi_{1,a_1(\lambda)}, \varphi_{2, a_2(\lambda)}, \varphi_{3, a_3(\lambda)},0,0,0),
$$
 represents a real-analytic family of static profiles for the sine-Gordon equation \eqref{sg2} on a tricrystal junction,  with 
 the profiles $\varphi_{j,a_j(\lambda)}$ satisfying the boundary condition in \eqref{antikink}. Hence we obtain, for $a_i=a_i(\lambda)$ and $\varphi_i=\varphi_{i,a_i(\lambda)}$, the following behavior:
\begin{enumerate}
\item[1)]  for $\lambda=-\frac{\pi}{2}(c_2+c_3-c_1)$ we obtain $a_1=a_2=a_3=0$,  $\varphi_2(0)=\varphi_3(0)=-\pi$, $\varphi_1(0)=\pi$, $\varphi_i''(0) =0$, $i=1,2, 3$. Thus, the profile of $(\varphi_1, \varphi_2,\varphi_3)$ represents one-half positive anti-kink  $\varphi_1$ and two-half negative kink solitons profiles  $\varphi_2$, $\varphi_3$ (connected in the vertex of the graph) such as Figure \ref{figTCZero} shows below; 

\begin{figure}[h]
\begin{center}
\subfigure[Kink/anti-kink profiles, $\varphi_j$.]{\label{figTCZeroProfiles}\includegraphics[scale=.55, clip=true]{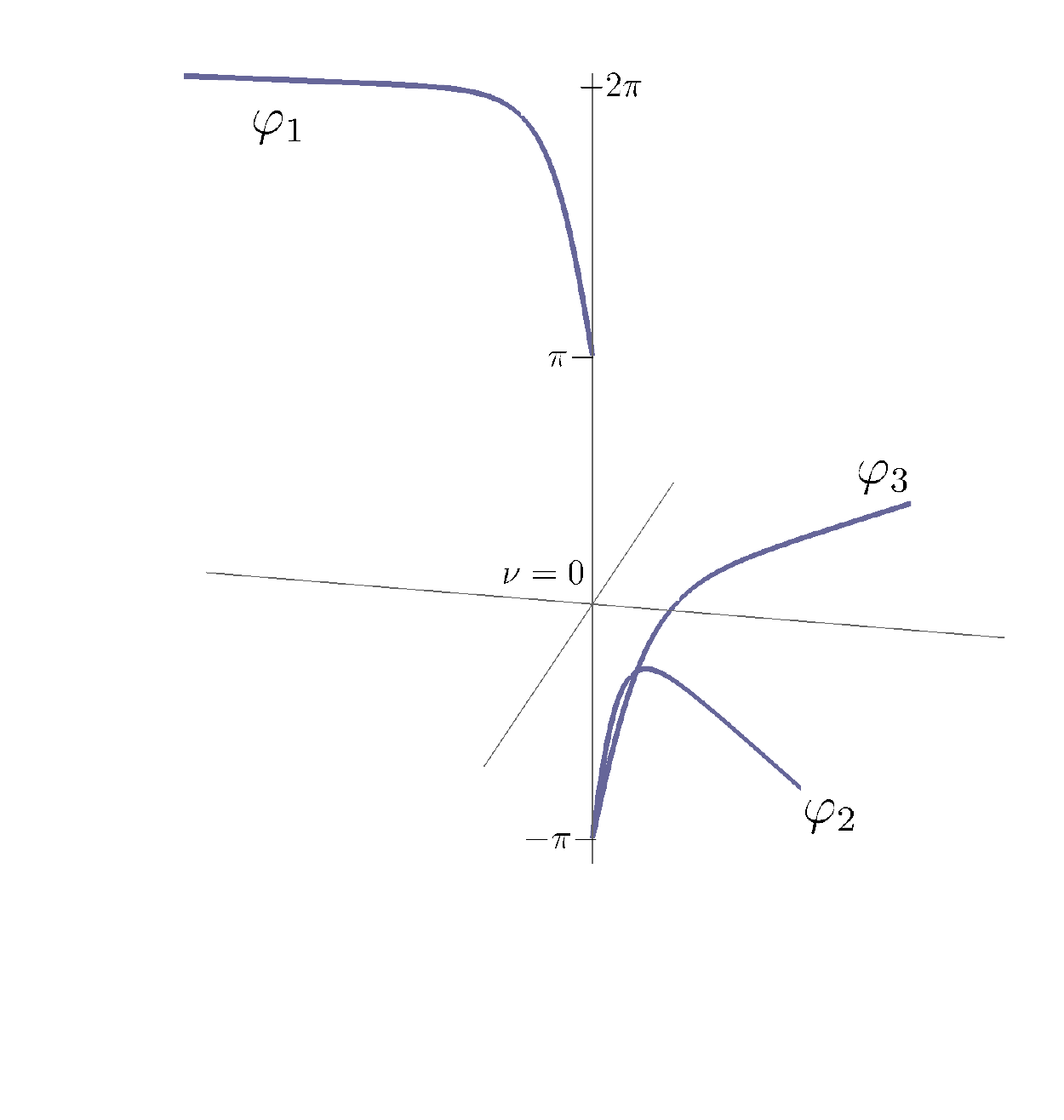}}
\subfigure[``Fluxons'', $\varphi_j'$.]{\label{figTCZeroDer}\includegraphics[scale=.55, clip=true]{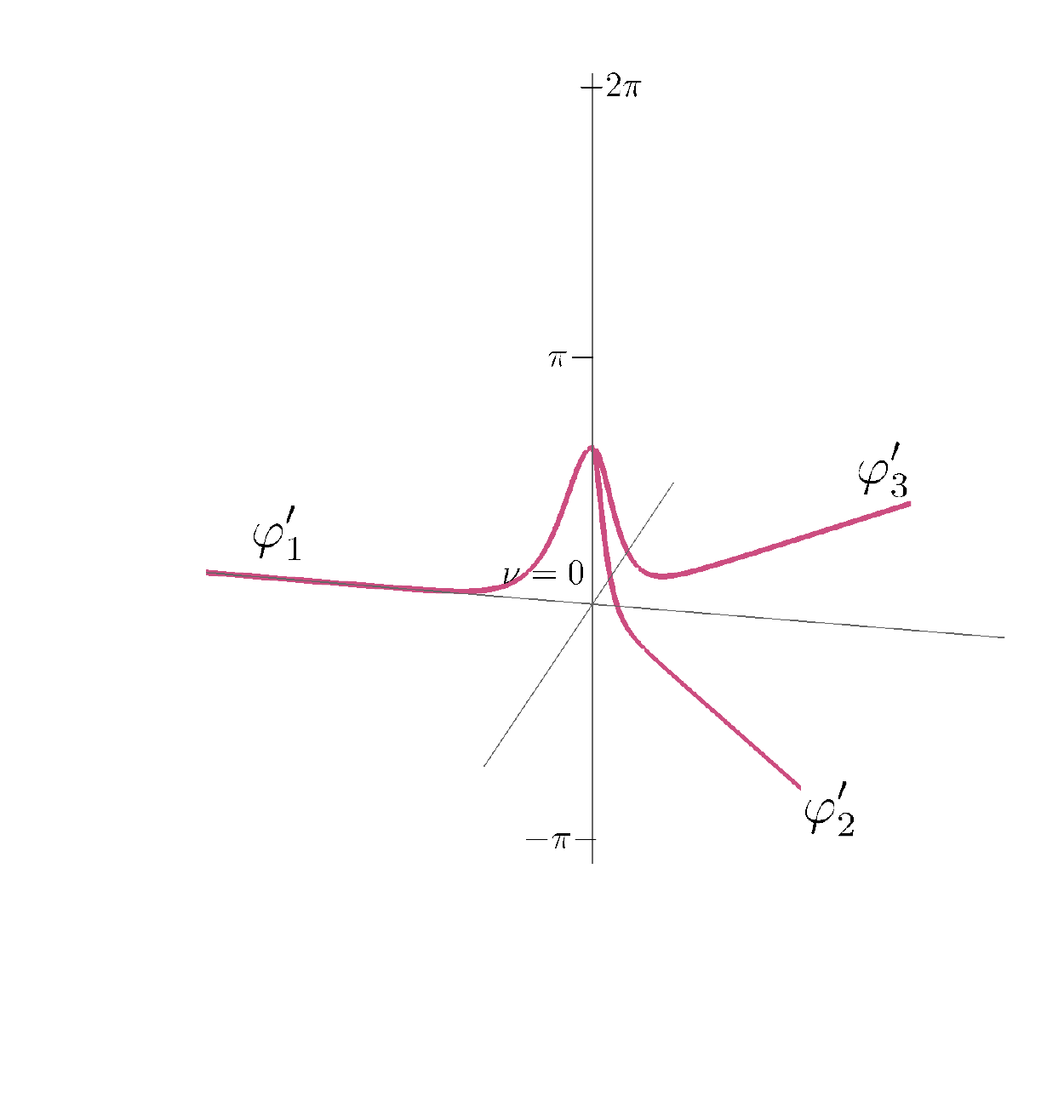}}
\end{center}
\caption{\small{Plots of the kink/anti-kink profiles \eqref{antikink} (panel (a), in blue) and their corresponding derivatives or ``fluxons'' (panel (b), in red) in the case where $c_j = 1$, $a_j = 0$, $j=1,2,3$, and $\lambda = - \tfrac{\pi}{2} (c_2+c_3-c_1) = - \tfrac{\pi}{2}$. Both plots are depicted in the same scale for comparison purposes (color online).}}\label{figTCZero}
\end{figure}

\item[2)] for  $\lambda\in (-\infty, -\frac{\pi}{2}(c_2+c_3-c_1))$ we obtain $a_1>0$, $\varphi_i''(a_1)=0$,  $i=1,2, 3$. Therefore, $\varphi_1(0)\in (0, \pi)$ and  $\varphi_2(0)=\varphi_3(0)\in (-\pi, -2\pi)$.
Thus, the profile of $(\varphi_1, \varphi_2,\varphi_2)$, looks similar to the one shown in Figure \ref{figTCBump} below (bump-profile type); 

\begin{figure}[h]
\begin{center}
\subfigure[Kink/anti-kink profiles, $\varphi_j$.]{\label{figTCBumpProfiles}\includegraphics[scale=.55, clip=true]{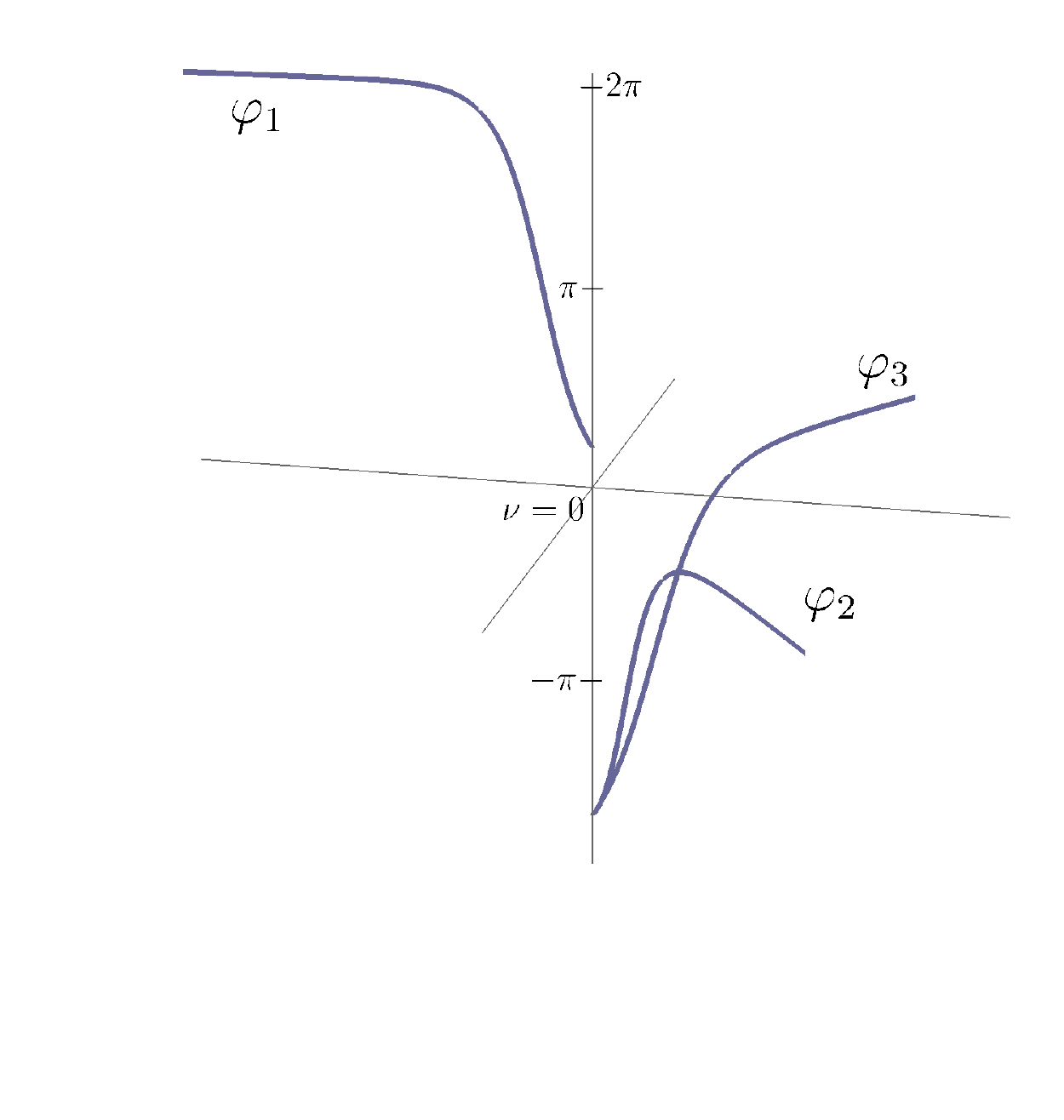}}
\subfigure[``Fluxons'', $\varphi_j'$.]{\label{figTCBumpDer}\includegraphics[scale=.55, clip=true]{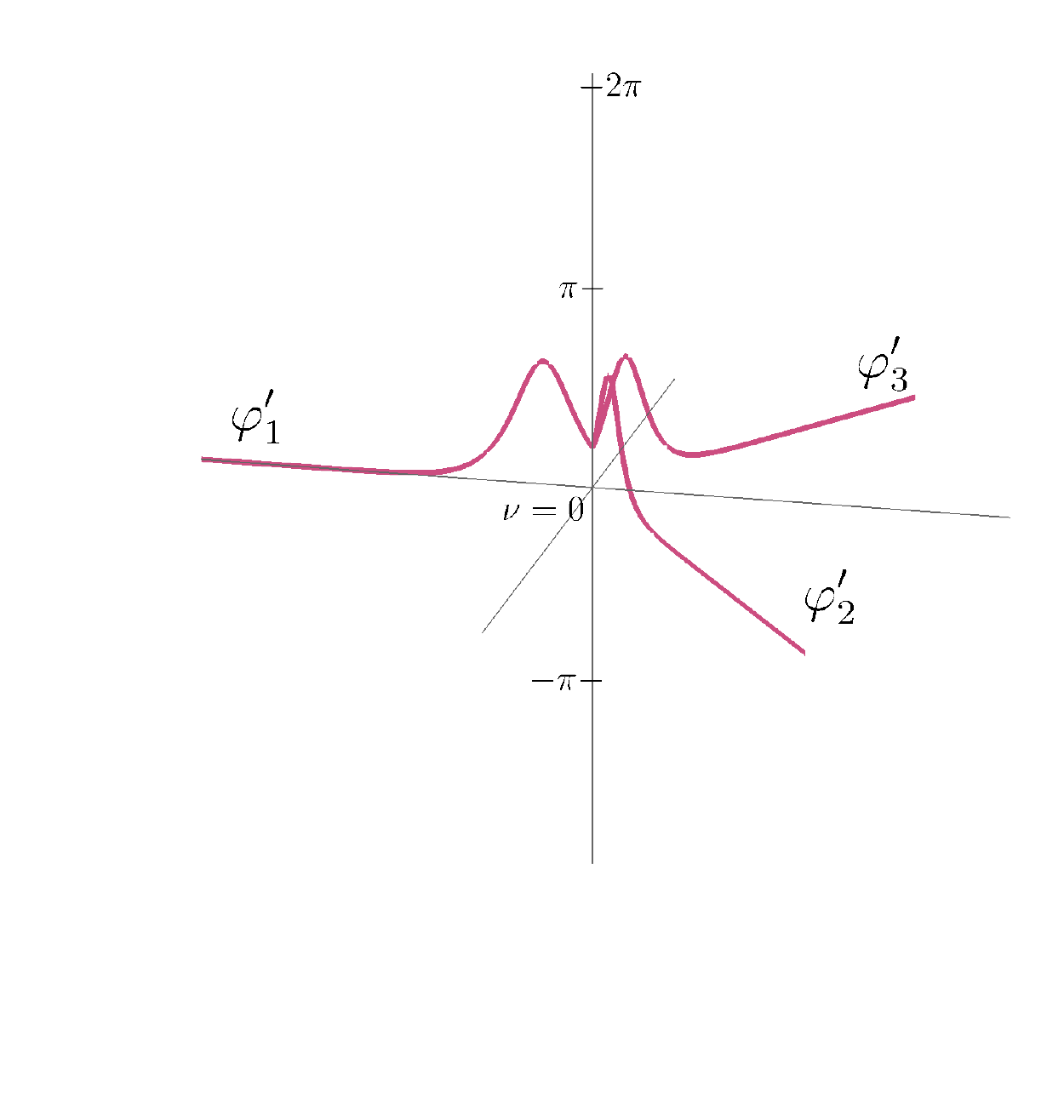}}
\end{center}
\caption{\small{Plots of the kink/anti-kink profiles \eqref{antikink} (panel (a), in blue) and their corresponding derivatives or ``fluxons'' (panel (b), in red) in the case where $c_j = 1$, $a_j = 1.8 > 0$, $j=1,2,3$, and $\lambda \in  (-\infty, - \tfrac{\pi}{2})$. Both plots are depicted in the same scale for comparison purposes (color online).}}\label{figTCBump}
\end{figure}

\item[3)] for   $\lambda\in (-\frac{\pi}{2}(c_2+c_3-c_1), +\infty)$ we obtain $a_1<0$ and 
 therefore   $ \varphi''_i<0$ for $i=1,2,3$,   Thus, the profile of $(\varphi_1, \varphi_2,\varphi_3)$ looks similar to that in Figure \ref{figTCTail} below (typical tail-profile). 
\end{enumerate}

\begin{figure}[h]
\begin{center}
\subfigure[Kink/anti-kink profiles, $\varphi_j$.]{\label{figTCTailProfiles}\includegraphics[scale=.55, clip=true]{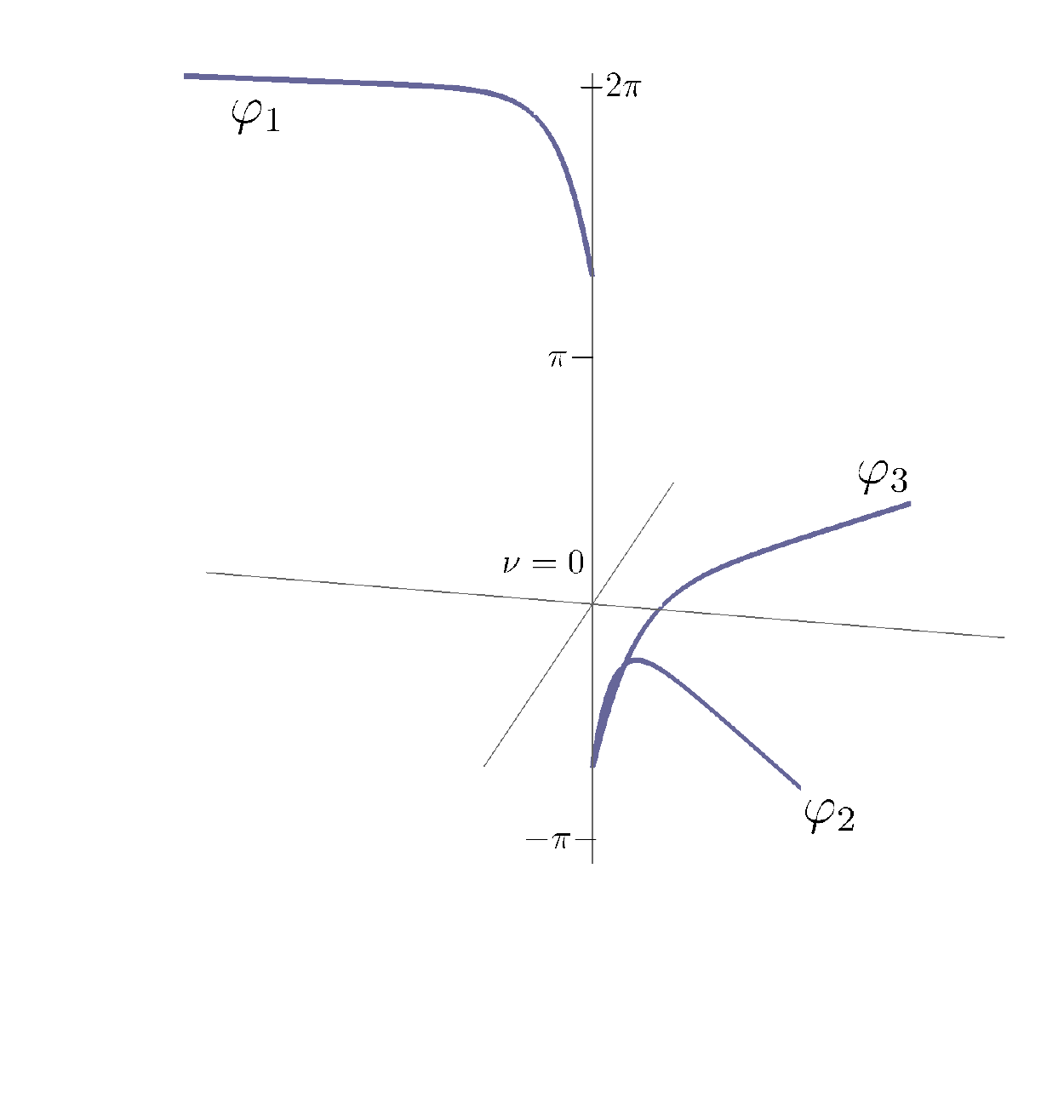}}
\subfigure[``Fluxons'', $\varphi_j'$.]{\label{figTCTailDer}\includegraphics[scale=.55, clip=true]{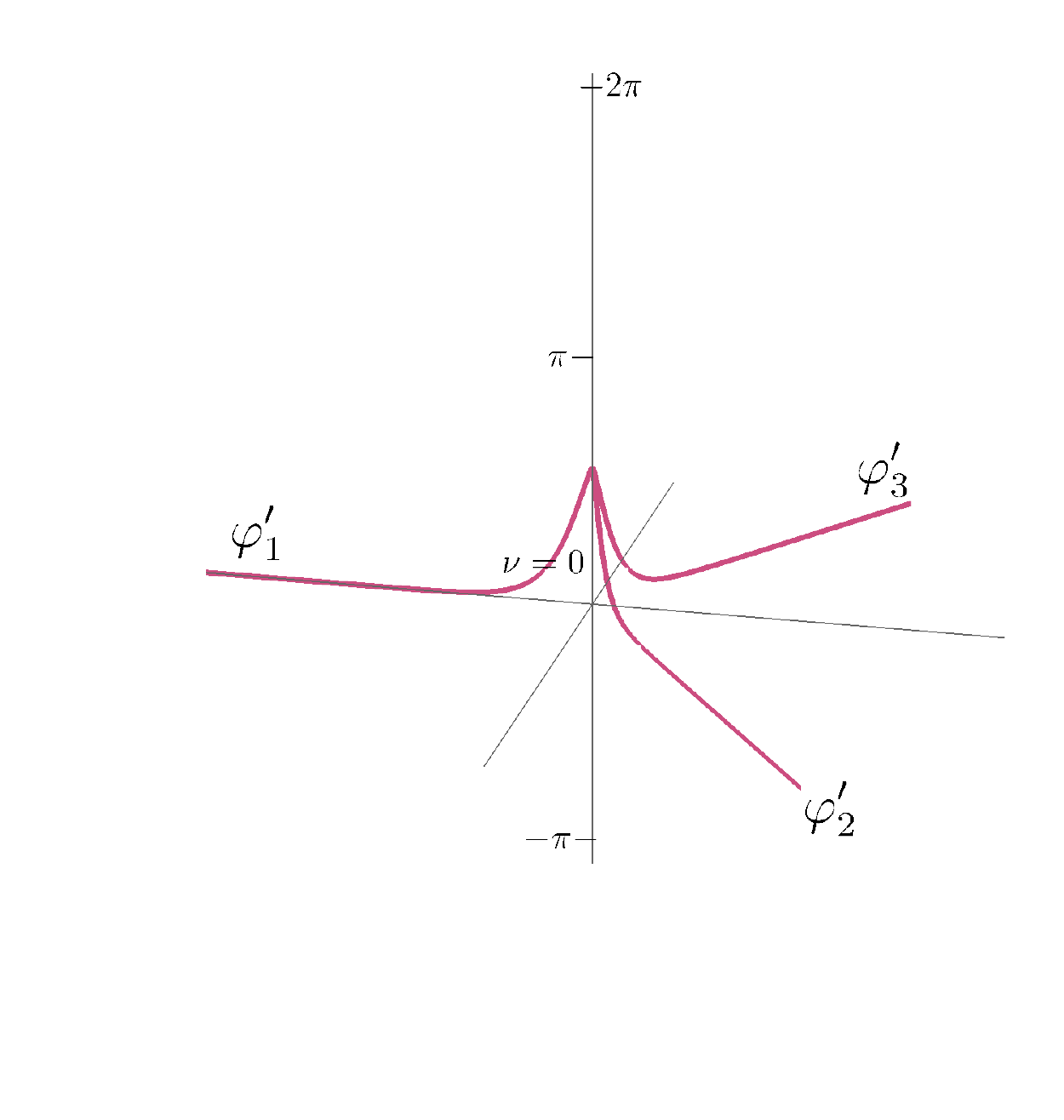}}
\end{center}
\caption{\small{Plots of the kink/anti-kink profiles \eqref{antikink} (panel (a), in blue) and their corresponding derivatives or ``fluxons'' (panel (b), in red) in the case where $c_j = 1$, $a_j = -0.5 < 0$, $j=1,2,3$, and $\lambda \in  (- \tfrac{\pi}{2}, \infty)$. Both plots are depicted in the same scale for comparison purposes (color online).}}\label{figTCTail}
\end{figure}

Here, our instability study of anti-kink/kink profiles will be in the case where the ``magnetic field'' or ``fluxon'' is continuous at the vertex of the tricrystal, namely, whenever $\varphi'_1(0)=\varphi'_2(0)=\varphi'_3($0) ($c_1=c_2=c_3$). It is to be observed that we have left open a possible instability study of other anti-kink/kink soliton configurations, for instance, profiles not satisfying the continuity property at zero for the components $\varphi_2, \varphi_3$ ($\frac{a_2}{c_2}\neq\frac{a_3}{c_3}$) and/or  the non-continuity of the magnetic field.  

The stability result for the stationary profiles, $\Phi_{\lambda,\delta'}=(\varphi_1,\varphi_2,\varphi_3, 0,0,0)$, with $\varphi_j=\varphi_{j, a_j(\lambda)}$ defined in \eqref{antikink}  is that established in Theorem \ref{0antimain}.

\subsubsection{Functional space for stability properties of the anti-kink/kink profile}
\label{akfunspa}
The natural framework space for studying  stability properties associated to the anti-kink/kink soliton profile $\Phi=(\varphi_j)_{j=1}^3$ described in the former subsection for the sine-Gordon model is $\mathcal X(\mathcal Y)= H^1_{\mathrm{loc}}(0,\infty)\bigoplus H^1(0, \infty)\bigoplus H^1(0, \infty)$. Thus we say that a flow $t\to (u(t), v(t))\in \mathcal X(\mathcal Y) \times L^2(\mathcal Y)$ is called  a \emph{perturbed solution} for the  anti-kink/kink profile $\Phi\in \mathcal X(\mathcal Y)$ if for $(P(t), Q(t))\equiv (u(t)-\Phi,v(t))$ we have that $ (P(t), Q(t))\in H^1(\mathcal Y)\times L^2(\mathcal Y)$ and  
$\bold z=(P,Q)^\top$ satisfies the following vectorial perturbed sine-Gordon model 
 \begin{equation}\label{akmodel}
 \begin{cases}
  \bold z_t=JE\bold z +F_1(\bold z)\\
 P(0)=u(0)-\Phi \in H^1(\mathcal Y),\\
 Q(0)=v(0)\in  L^2(\mathcal Y),
  \end{cases}
 \end{equation} 
  where for $P=(p_1, p_2, p_3)$ we have
   \begin{equation}
  F_1(\bold z)=\left(\begin{array}{cc}  0\\  0   \\  0   \\  \sin(\varphi_1)-\sin (p_1+\varphi_1)   \\  \sin(\varphi_2)-\sin (p_2+\varphi_2)   \\  \sin(\varphi_3)-\sin (p_3+\varphi_3)  \end{array}\right).
  \end{equation} 
Then, the stability analysis of the stationary anti-kink/kink $\Phi_{\lambda, \delta'}=(\varphi_1,\varphi_2,\varphi_3,0,0,0)$ by the sine-Gordon model on $\mathcal X(\mathcal Y) \times L^2(\mathcal Y)$ reduces to studying the stability properties of the trivial solution $(P,Q)=(0,0)$ for the  linearized  model associated to \eqref{akmodel} around $(P,Q)=(0,0)$. Thus, via Taylor's Theorem  we obtain the  linearized system in \eqref{stat4} but with the Schr\"odinger diagonal operator $\mathcal L$ in \eqref{stat5} now determined by the anti-kink/kink profile $\Phi=(\varphi_j)$. We denote this operator by $\mathcal L_{\lambda}$ in \eqref{deriva1} below, with the $\delta'$-interaction domain $D(\mathcal L_{\lambda})$ in \eqref{tricry2}. In this form, we can apply \emph{ipsis litteris} the  semi-group theory results in subsection 3.1.1  to the operator $JE$ and to the local well-posedness problem in $H^1(\mathcal Y)\times L^2(\mathcal Y)$ for the vectorial perturbed sine-Gordon model \eqref{akmodel}. Lastly, we note that the anti-kink/kink profile $\Phi\in \mathcal X(\mathcal Y)$, but $\Phi'\in H^2(\mathcal Y)$.

\subsubsection{The spectral study in the anti-kink/kink case}

In this subsection we provide the necessary spectral information for the family of self-adjoint operators $(\mathcal{L}_\lambda, D(\mathcal{L}_\lambda))$, where
\begin{equation}\label{deriva1}
\mathcal{L}_\lambda =\Big (\Big(-c_j^2\frac{d^2}{dx^2}+\cos(\varphi_j)
\Big)\delta_{j,k} \Big ),\quad \l1\leqq j, k\leqq 3,
\end{equation}
associated to the anti-kink/kink solutions $(\varphi_1, \varphi_2, \varphi_3)$ determined in the previous subsection. Here $D(\mathcal{L}_\lambda)$ is the $\delta'$-interaction domain defined  in \eqref{2trav23}.

By completeness and convenience of the reader, we prove the following  results associated to the kernel of the operators $\mathcal{L}_\lambda$ for all values of  $c_1, c_2, c_3$ and the admissible $\lambda$-values. We note  that these results are expected, due to the break in the translation symmetry of solutions for the sine-Gordon model on tricrystal configurations. 

\begin{proposition}\label{antimain}
Let $c_1, c_2, c_3>0$ and $\lambda\in  (-\infty,+\infty)$. Then, $\ker( \mathcal{L}_\lambda )=\{\mathbf{0}\}$ for all $\lambda\neq -\frac{\pi}{2}(c_2+c_3-c_1)$. For $\lambda=-\frac{\pi}{2}(c_2+c_3-c_1)$ we have $\dim (\ker( \mathcal{L}_\lambda) )=2$. Moreover, for all $\lambda$ we obtain $\sigma_{\mathrm{ess}}(\mathcal{L}_\lambda)=[1,+\infty)$.
\end{proposition}

\begin{proof} Let $\bold u=(u_1, u_2, u_3)\in D(\mathcal{L}_\lambda)$ and $\mathcal{L}_\lambda \bold{u}=\bold{0}$. Then, since $(\varphi'_1, \varphi'_3, \varphi'_2)\in H^2(\mathcal Y)$, it follows from Sturm-Liouville theory on half-lines that
\begin{equation}\label{anti2spec}
 u_j(x) = \alpha_j\phi'_j(x), \;\;x>0,\;\; j=1,2,3,
\end{equation}
for some $\alpha_j$, $j = 1,2,3$, real constant. In what follows, we assume $\lambda\neq -\frac{\pi}{2}(c_2+c_3-c_1)$. Thus, from \eqref{tricry2}  and by supposing $ \alpha_1\neq 0$ we have
\begin{equation} 
\label{reker0}
\frac{\alpha_2}{\alpha_1}=\frac{c_1}{c_2} \frac{\varphi''_1(0)}{ \varphi''_2(0)}=\frac{c_2}{c_1},\;\;\;\; \frac{\alpha_3}{\alpha_1}=\frac{c_1}{c_3} \frac{\varphi''_1(0)}{ \varphi''_3(0)}=\frac{c_3}{c_1}
\end{equation}
and
\begin{equation} 
\label{reker}
\lambda=\frac{1}{\alpha_1}\Big(\sum_{j=1}^3\alpha_j \Big) \frac{\varphi'_1(0)}{ \varphi''_1(0)}=(c_1+c_2+c_3) \frac{\cosh^2(a_1/c_1)}{\sinh(a_1/c_1)}.
\end{equation}

Now, we consider the following cases:
\begin{enumerate}

\item[a)] Suppose $a_1<0$ ($\lambda\in (-\frac{\pi}{2}(c_2+c_3-c_1), +\infty)$): then, since
\begin{equation}\label{inequa}
\frac{\cosh^2(a_1/c_1)}{\sinh(a_1/c_1)}\leqq -2,\;\;\text{for all}\;\; a_1<0,
\end{equation}
we have, from \eqref{reker}, $-\frac{\pi}{2}(c_2+c_3-c_1)<-2(c_1+c_2+c_3)$, which is a contradiction.  So, we need to have $0=\alpha_1=\alpha_2=\alpha_3$ and therefore $ \bold{u}=\bold{0}$.

\item[b)] Suppose $a_1>0$ ($\lambda\in (-\infty, -\frac{\pi}{2}(c_2+c_3-c_1)$): then, from  \eqref{reker} we have 
$\lambda> 2(c_1+c_2+c_3)$.  Thus for the case $c_2+c_3-c_1\geqq 0$ we obtain immediately a contradiction. Now, for  $c_2+c_3-c_1<0$ is sufficient to study the case $0<\lambda<  -\frac{\pi}{2}(c_2+c_3-c_1)$. Indeed, since  $2(c_1+c_2+c_3)> -\frac{\pi}{2}(c_2+c_3-c_1)$ we have again a contradiction. So, we need to have $0=\alpha_1=\alpha_2=\alpha_3$ and therefore $ \bold{u}=\bold{0}$.
\end{enumerate}

Now, suppose that $\lambda=-\frac{\pi}{2}(c_2+c_3-c_1)$.  In this case the Kirchhoff's condition for $\bold u$, $\varphi'_1(0)\neq 0$ and $\varphi''_1(0)=0$ we get the relation $\alpha_1+\alpha_2+\alpha_3=0$. Therefore
$$
(u_1, u_2, u_3)=\alpha_2(-\varphi'_1, \varphi'_2,0)+ \alpha_3(-\varphi'_1, 0,  \varphi'_3).
$$
Since $(-\varphi'_1, \varphi'_2,0), (-\varphi'_1, 0,  \varphi'_3)\in D(\mathcal L_\lambda)$ we obtain that $\dim(\ker(\mathcal L_\lambda))=2$.

The statement $\sigma_{\mathrm{ess}}(\mathcal{L}_\lambda)=[1,+\infty)$ is an immediate consequence of Weyl's Theorem because of $\lim_{x\to +\infty} \cos(\varphi_1(x))=1=\lim_{x\to +\infty} \cos(\varphi_j(x))$. This finishes the proof.
\end{proof}

\begin{proposition}
\label{anti5main}
Let $c_1=c_2=c_3$ and $\lambda\in [-\frac{\pi}{2}c_1, 0)$. Then $n( \mathcal{L}_\lambda )=1$.
\end{proposition}

\begin{proof}  Similarly as in the proof of Proposition \ref{2crystal} (see formula \eqref{7spec}), we obtain via the extension theory that $n( \mathcal{L}_\lambda )\leqq 1$ for all $\lambda\in [-\frac{\pi}{2}c_1, +\infty)$  because of $\varphi_j''(0)\leqq 0$ for $j=1,2,3$.

Now we show that $n(\mathcal{L}_\lambda)\geqq 1$ for $\lambda\in [-\frac{\pi}{2}c_1, 0)$. Indeed, we consider the following quadratic form $\mathcal Q_\lambda$ associated to $(\mathcal{L}_\lambda, D(\mathcal{L}_\lambda))$ for  $\Lambda =(\psi_i)\in H^1(\mathcal Y)$,
\begin{equation}
\label{antiqua}
\mathcal Q(\Lambda)=\frac{1}{\lambda} \Big(\sum_{j=1}^3c_j \psi_j(0)\Big)^2 +\sum_{j=1}^3 \int_0^{\infty}c^2_j(\psi'_j)^2+\cos(\varphi_j) \psi_j^2 dx.
\end{equation}
Next, for $\Lambda_1=(\varphi'_1, \varphi'_2, \varphi'_3)\in H^1(\mathcal Y)$ we obtain  from the equalities  $-\varphi'''_j + \cos(\varphi_j)\varphi'_j=0$, $c_1\varphi'_1(0)=c_2\varphi'_2(0)=c_3\varphi'_3(0)$, and from integration by parts, the relation
\begin{equation}\label{2antiqua}
\mathcal Q_\lambda(\Lambda_1)=\frac{9c_1^2}{\lambda} [\varphi'_1(0)]^2 -c_1\varphi'_1(0) \sum_{j=1}^3 c_j \varphi''_j(0).
\end{equation}
Thus, for $\lambda=-\frac{\pi}{2}c_1$ we have $\varphi''_j(0)=0$ and therefore $\mathcal Q_\lambda(\Lambda_1)<0$. Thus $n(\mathcal{L}_{-\frac{\pi}{2}c_1})=1$. Now, for $\lambda\in (-\frac{\pi}{2}c_1, 0)$ we have $\varphi''_1(0)<0$ ($a_1<0$) and so since $\varphi''_1(0)=\varphi''_2(0)=\varphi''_3(0)$ we get $\mathcal Q_\lambda(\Lambda_1)<0$ if and only if 
\begin{equation}
3\frac{\varphi'_1(0)}{ \varphi''_1(0)}=3c_1\frac{\cosh^2(a_1)}{\sinh(a_1)}<\lambda, \;\;\text{for}\;\; a_1<0,
\end{equation}
which is true by \eqref{inequa}. Therefore, $n(\mathcal{L}_{\lambda})=1$ for $\lambda\in (-\frac{\pi}{2}c_1, 0)$.
\end{proof}

\begin{remark}\label{antifora}
For the case $\lambda\in [0, +\infty)$ in Proposition \ref{anti5main}, it was not possible to show (in an easy way) that the quadratic form $\mathcal Q_\lambda$ in \eqref{antiqua} has a negative direction. But we will see in the following, via analytic perturbation approach, that we still have $n( \mathcal{L}_\lambda )=1$.
\end{remark}

\begin{proposition}
\label{antimain4}
Let  $c_1=c_2=c_3$ and  $\lambda\in [0, +\infty)$. Then $n( \mathcal{L}_\lambda )=1$.
\end{proposition}

\begin{proof} We will use analytic perturbation theory. Initially, from subsection 3.2.1 we have that  $\lambda\in (-\infty ,+\infty)\to a_1(\lambda)$ represents a  real-analytic   mapping function and so  from the relations $\varphi_{1, a_1(\lambda)}-\varphi_{1, a_1(0)}\in L^2(0, +\infty)$ for every $\lambda$ and $\|\varphi_{1, a_1(\lambda)}-\varphi_{1, a_1(0)}\|_{H^1(0, +\infty)} \to 0$ as $\lambda\to 0$, we obtain for $\Phi_{a_1(\lambda)}=(\varphi_{1, a_1(\lambda)}, \varphi_{3, a_1(\lambda)}, \varphi_{2, a_1(\lambda)})$ (where we have used that $a_1(\lambda)=a_2(\lambda)=a_3(\lambda)$) the convergence
$$
\|\Phi_{a_1(\lambda)}- \Phi_{a_1(0)}\|_{H^1(\mathcal Y)}\to 0\;\;\text{as}\;\;\lambda\to 0.
$$ 

Thus, we obtain that  $\mathcal L_{\lambda}$ converges to $\mathcal L_{0}$ as $\lambda\to 0$ in the generalized sense. Indeed, denoting $W_\lambda=\Big(\cos (\varphi_{j, a_1(\lambda)})\delta_{j,k}\Big)$
we obtain 
\begin{equation*}
 \begin{split}
 \widehat{\delta}(\mathcal L_\lambda, \mathcal L_{0} )&=\widehat{\delta}(\mathcal L_0 + (W_\lambda-W_{0}), \mathcal L_{0})\leqq \|W_\lambda-W_{0}\|_{L^2 (\mathcal Y)}\to 0,\qquad\text{as}\;\;\lambda\to 0,
\end{split}
\end{equation*}
where $\widehat{\delta}$ is  the gap metric  (see \cite[Chapter IV]{kato}). 

Now, we denote by $N=n(\mathcal L_{0})$ the Morse-index for $\mathcal L_{0}$ (from the proof of Proposition \ref{anti5main}, $N\leqq 1$). 
Thus, from Proposition \ref{antimain} we can separate the spectrum $\sigma(\mathcal L_{0})$ of $\mathcal L_{0}$   into two parts, $\sigma_0=\{\gamma: \gamma<0\}\cap \sigma (\mathcal L_{0})$ and $\sigma_1$ by a closed curve  $\Gamma$ belongs to the resolvent set of $\mathcal L_{0}$ with $0\in \Gamma$ and   such that $\sigma_0$ belongs to the inner domain of $\Gamma$ and $\sigma_1$ to the outer domain of $\Gamma$. Moreover, $\sigma_1\subset [\theta_{0}, +\infty)$ with $\theta_{0}=\inf\{\theta: \theta\in \sigma(\mathcal L_0),\; \theta>0\}>0$ (we recall that $\sigma_{\mathrm{ess}}(\mathcal L_0)=[1,+\infty)$). Then, by  \cite[Theorem 3.16, Chapter IV]{kato}, we have  $\Gamma\subset \rho(\mathcal L_{\lambda})$ for $\lambda\in [-\delta_1, \delta_1]$ and $\delta_1>0$ small enough. Moreover, $\sigma(\mathcal L_{\lambda})$ is likewise separated by $\Gamma$ into two parts so that the part of $\sigma (\mathcal L_{\lambda})$ inside $\Gamma$ will consist of a negative eigenvalue with exactly total (algebraic) multiplicity equal to $N$. Therefore, by Proposition \ref{anti5main} we need to have $n(\mathcal L_{\lambda})=N=1$ for $\lambda\in [-\delta_1, \delta_1]$.

Next, using a classical continuation argument based on the Riesz-projection, we can see that $n( \mathcal{L}_\lambda )=1$ for all $\lambda\in  [0, +\infty)$ (see \cite{AnPl-delta}).  This finishes the proof.
\end{proof}

The following result gives a precise value for the Morse-index of the operator $ \mathcal{L}_{\lambda} $, with $\lambda\in (-\infty, -\frac{\pi}{2}c_1)$,  when we consider the  domain $\mathcal C\cap D (\mathcal{L}_{\lambda})$, $\mathcal C=\{(u_j)_{j=1}^3\in L^2(\mathcal Y): u_1=u_2=u_3\}$. This strategy  will allow the use of the linear instability framework established in section 2 above.  

\begin{proposition}
\label{Appanti}
Let $c_1=c_2=c_3$ and $\lambda\in (-\infty, -\frac{\pi}{2}c_1]$.  Consider  $\mathcal B=\mathcal C \cap D (\mathcal{L}_{\lambda})$.  Then $\mathcal{L}_{\lambda}: \mathcal B\to \mathcal C$ is well defined. Moreover,  $\ker( \mathcal{L}_\lambda |_{\mathcal B})=\{0\}$ and $n( \mathcal{L}_{\lambda} |_{\mathcal B} )=1$.
\end{proposition}

\begin{proof} Initially, since $\varphi_1-2\pi=\varphi_2=\varphi_3$ and $\cos(\varphi_1) = \cos(\varphi_1-2\pi) = \cos(\varphi_2) = \cos(\varphi_3)$ follows immediately that, for $\bold u=(u_j)_{j=1}^3\in \mathcal B$, we have $ \mathcal{L}_{\lambda}\bold u\in \mathcal C$.

Next, from the proof of Proposition \ref{antimain} for $\lambda_0=-\frac{\pi}{2}c_1$ we have $\ker(\mathcal{L}_{\lambda_0} |_{\mathcal B} )=\{\bold 0\}$. Moreover, the anti-kink/kink profile for $\lambda_0$, $\Phi_0=(\varphi_{1, 0}, \varphi_{2,0}, \varphi_{2,0})$, satisfies that $\Phi'_0=(\varphi'_{1,0}, \varphi'_{2,0}, \varphi'_{2,0})\in \mathcal C\cap H^1(\mathcal Y)$. Then, from \eqref{2antiqua} we know that the quadratic form $Q_{\lambda_0}$ associated to $(\mathcal{L}_{\lambda_0}, \mathcal B)$ satisfies
$Q_{\lambda_0}(\Phi'_0)<0$. Therefore, $n( \mathcal{L}_{\lambda_0} |_{\mathcal B} )=1$. Hence, by using a similar strategy as in the proof 
of Proposition  \ref{antimain4} we obtain that $n( \mathcal{L}_{\lambda} |_{\mathcal B} )=1$ for $\lambda\in (-\infty, -\frac{\pi}{2}c_1)$. This finishes the proof.
\end{proof}

\begin{proof}[Proof of Theorem \ref{0antimain}] We consider initially the case $\lambda\in (-\frac{\pi}{2}c_1, +\infty)$. From Propositions \ref{antimain}, \ref{anti5main} and \ref{antimain4} we have $\ker( \mathcal{L}_\lambda)=\{0\}$ and $n( \mathcal{L}_\lambda)=1$. Moreover, from subsection \ref{akfunspa} and based in Proposition \ref{coro1} we verify Assumption $(S_1)$ for $J\mathcal E$ on $H^1(\mathcal Y)\times L^2(\mathcal Y)$ in the linear instability criterion in subsection 3.1. Thus, from Theorem \ref{crit} follows the linear instability property of the stationary anti-kink/kink soliton profile $\Pi_{\lambda, \delta'}$. 

For the case $\lambda\in (-\infty,-\frac{\pi}{2}c_1]$ we need to adjust our instability criterion in section 2 to the space $\mathcal R=(\mathcal C\cap H^1(\mathcal Y))\times \mathcal C$ determined by \eqref{akmodel}. Thus, initially we need to see that $J\mathcal E$ is the generator of a $C_0$-semigroup on the restricted subspace $\mathcal R$ (assumption $(S_1)$). Indeed, this is a consequence of the mapping $J\mathcal E: (\mathcal C\cap D(\mathcal L_\lambda))\times (\mathcal C\cap H^1(\mathcal Y))\to (\mathcal C\cap H^1(\mathcal Y))\times \mathcal C$ is well defined and  by using the same strategy as in Angulo and Plaza \cite{AnPl-deltaprime} (subsection 3.1.1 and Proposition 3.11). Now, from Proposition \ref{Appanti} and Theorem \ref{crit} we finish the proof.

\end{proof}


\section*{Acknowledgements}
J. Angulo  was supported in part by CNPq/Brazil Grant and by FAPERJ/Brazil program PRONEX-E - 26/010.001258/2016.  

\appendix
\section{Appendix }
\label{secApp}

For the sake of completeness, in this section we develop the extension theory of symmetric operators suitable for our needs. For further information on the subject the reader is referred to the monographs by Naimark \cite{Nai67,Nai68}. The following classical result, known as  \emph{the von-Neumann decomposition theorem}, can be found in \cite{RS1, RS2}.
 
\begin{theorem}\label{d5} 
Let $A$ be a closed, symmetric operator, then
\begin{equation}\label{d6}
D(A^*)=D(A)\oplus\mathcal N_{-i} \oplus\mathcal N_{+i}.
\end{equation}
with $\mathcal N_{\pm i}= \ker (A^*\mp iI)$. Therefore, for $u\in D(A^*)$ and $u=x+y+z\in D(A)\oplus\mathcal N_{-i} \oplus\mathcal N_{+i}$,
\begin{equation}\label{d6a}
A^*u=Ax+(-i)y+iz.
\end{equation}
\end{theorem}

\begin{remark} The direct sum in (\ref{d6}) is not necessarily orthogonal.
\end{remark}

The following propositions provide a strategy for estimating the Morse-index of the self-adjoint extensions (see Naimark \cite{Nai68}).

\begin{proposition}\label{semibounded}
Let $A$  be a densely defined lower semi-bounded symmetric operator (that is, $A\geq mI$)  with finite deficiency indices, $n_{\pm}(A)=k<\infty$,  in the Hilbert space $\mathcal{H}$, and let $\widehat{A}$ be a self-adjoint extension of $A$.  Then the spectrum of $\widehat{A}$  in $(-\infty, m)$ is discrete and  consists of, at most, $k$  eigenvalues counting multiplicities.
\end{proposition}

\begin{proposition}\label{11}
	Let $A$ be a densely defined, closed, symmetric operator in some Hilbert space $H$ with deficiency indices equal  $n_{\pm}(A)=1$. All self-adjoint extensions $A_\theta$ of $A$ may be parametrized by a real parameter $\theta\in [0,2\pi)$ where
	\begin{equation*}
	\begin{split}
	D(A_\theta)&=\{x+c\phi_+ + \zeta e^{i\theta}\phi_{-}: x\in D(A), \zeta \in \mathbb C\},\\
	A_\theta (x + \zeta \phi_+ + \zeta e^{i\theta}\phi_{-})&= Ax+i \zeta \phi_+ - i \zeta e^{i\theta}\phi_{-},
	\end{split}
	\end{equation*}
	with $A^*\phi_{\pm}=\pm i \phi_{\pm}$, and $\|\phi_+\|=\|\phi_-\|$.
\end{proposition}

\def\cprime{$'\!\!$} \def\cprimel{$'\!$}



\end{document}